\documentclass[twoside,11pt]{amsart}
\usepackage[all]{xy}
\CompileMatrices
\usepackage{amsfonts}
\usepackage{amssymb}
\usepackage{amsthm}
\usepackage{amsmath}
\usepackage{ifthen}
\usepackage{enumitem}
\usepackage[usenames]{color}
\usepackage{comment}

 \newtheorem{thm}{Theorem}[section]
 \newtheorem{prop}[thm]{Proposition}
 
 \newtheorem{lem}[thm]{Lemma}

\theoremstyle{definition}
\newtheorem{defn}[thm]{Definition}

\theoremstyle{remark}
\newtheorem{rem}[thm]{Remark}

\newcommand{\Z}{\mathbb{Z}}
\newcommand{\Q}{\mathbb{Q}}
\newcommand{\R}{\mathbb{R}}
\newcommand{\C}{\mathbb{C}}

\newcommand{\GL}{\mathrm{GL}}

\newcommand{\diag}{\mathrm{diag}}

  1

\newcommand{\transpose}[1]{\text{$^t\!#1$}}
\newcommand{\G}{\ifmmode {\mathcal{G}}\else${\mathcal{G}}$\ \fi}

\def\sectionnam{\@empty}
\def\subsectionnam{\@empty}

\begin{document}

\title[Algebraicity of $L$-values attached to Quaternionic modular forms.] % end with percent
{Algebraicity of $L$-values attached to Quaternionic modular forms.}
%\newline
%\center{\textbf{Preliminary Version}}}

\author{Thanasis Bouganis and Yubo Jin}
\address{Department of Mathematical Sciences\\ Durham University\\
South. Rd.\\ Durham, DH1 3LE, U.K..}
\email{athanasios.bouganis@durham.ac.uk, yubo.jin@durham.ac.uk}

\maketitle 
\begin{abstract}
In this paper we prove the algebraicity of some $L$-values attached to quaternionic modular forms. We follow the rather well established path of the doubling method. Our main contribution is that we include the case where the corresponding symmetric space is of non-tube type. We make various aspects very explicit such as, the doubling embedding, coset decomposition, and the definition of algebraicity of modular forms via CM points. 
\end{abstract}

\section{Introduction}

Special values of $L$-functions attached to automorphic forms have a long history in modern number theory. Their importance is difficult to overestimate and for this reason they have been the 
subject of intense study in recent decades. There is no doubt that it is  important to study $L$-values of automorphic forms whose underlying symmetric space does not have hermitian structure (for example automorphic forms for $\mathrm{GL}_n$), however in this paper we will be dealing with a kind of automorphic form where the corresponding symmetric space has a hermitian structure. To go a bit further we now introduce some notation.

 Let $\mathbb{D}$ be a division algebra over $\Q$ and $V$ a free left $\mathbb{D}$-module of finite rank. Denote $\mathrm{End}(V,\mathbb{D})$ be the ring of $\mathbb{D}$-linear endomorphism of $V$ and $\mathrm{GL}(V,\mathbb{D})=\mathrm{End}(V,\mathbb{D})^{\times}$. For a nondegenerate hermitian (or skew-hermitian) form $\langle \,,\rangle:V\times V\to D$, we define a generalized unitary group
\[
G:=G_n:=\{g\in\mathrm{GL}(V,\mathbb{D}):\langle gx,gy\rangle=\langle x,y\rangle\}.
\] 
One can define automorphic forms associated to such group as in \cite{Borel}. These can be seen as functions on a symmetric space $G(\R)/K$, where $K$ is a maximal compact subgroup of $G(\mathbb{R})$. In addition when the associated symmetric space can be given a hermitian structure, one can define holomorphic automorphic forms which is what we refer as modular forms in this paper. The symmetric spaces $G(\R)/K$ have been classified \cite[Chapter X]{Helgason} (see also \cite{LKW}). In particular there are four infinite families of irreducible hermitian symmetric spaces of non compact type\\

(A) $\{z\in\C_m^n:zz^{\ast}<1_n\}$,\\
(B) $\{z\in\C^n:z^{\ast}z<1+|\transpose{z}z/2|^2<2\}$,\\
(C) $\{z\in\C_n^n:\transpose{z}=z,z^{\ast}z<1_n\}$,\\
(D) $\{z\in\C_n^n:\transpose{z}=-z,z^{\ast}z<1_n\}$.\\

The spaces above are the so-called bounded realisations of the symmetric spaces, and one can with the use of the Cayley transform show that are biholomorphic to unbounded domains. 
For example, when $\mathbb{D}=\Q$ and $\langle\, ,\,\rangle$ is skew-hermitian, then $G$ is the symplectic group and we have the notion of Siegel modular forms defined over symmetric spaces of type C.  The unbounded realisation is the classical Siegel upper space. When $\mathbb{D}$ is an imaginary quadratic field, $G$ is the unitary group and we have the notion of Hermitian modular forms defined over symmetric spaces of type A.  For these two types of domains (and their groups) there has been an intensive study on the algebraic properties of their attached special $L$-values. We will not cite here the vast literature that has grown in the past few decades or so, we will only mention here the book \cite{Sh00}, the more recent article \cite{P} and the references there in for a more complete account of the Siegel case, and the work of Harris \cite{Harris94}  in the Hermitian modular forms case.

The focus of this paper is on the domains of type D above.  This domain arises when we select $\mathbb{D}$ to be a definite quaternion algebra  and the form $\langle\, ,\,\rangle$ skew hermitian (see next section for details). There are already some works for these modular forms, for example \cite{B,K,U,Y}, but it is fair to say that these modular forms are not as intensively studied as the Siegel or Hermitian ones. Even more importantly most, if not all, of the works are restricted to the case when the dimension of $V$ is even. The importance of this restriction is related to the unbounded realisation of the corresponding symmetric domain. In particular when $n$ is even, the unbounded domain is biholomorphic to a tube domain, or what is usually called a domain of Siegel Type I. When $n$ is odd the domain is not any more of tube type (a similar aspect is seen also for Hermitian modular forms in the non-split case $U(n,m)$ with $n \neq m$). The significance of this distinction will become clear later in this paper, since the non-tube case is considerably more technical. For example, as we will see, the notion of an algebraic modular form cannot be the usual one (algebraic Fourier coefficients) or the doubling embedding which is needed in the doubling method is considerably more complicated to write explicitly in the unbounded realisation. 

As we have indicated we will be studying the special values of L-functions by using the doubling method of Garrett, Shimura, Piatetski-Shapiro and Rallis. Without going here into details but referring later to the paper, the key idea is to obtain an integral representation relating the L-function to the pull-back of a Siegel-type Eisenstein series. Then the analytic and algebraic properties of the L-function can be studied from those of the Eisenstein series. The latter is well-understood thanks to the rather explicitly known Fourier expansion.

Our starting point is a cuspidal Hecke eigenform $\mathbf{f} \in\mathcal{S}_k(K_1(\mathfrak{n}))$. We then consider two copies of our group $G_n$ with an embedding $G_{n}\times G_{n}\to G_N$ with $N=2n$ and hence $G_N$ splits (see section 4 for notation). For $P_N$ the Siegel parabolic subgroup of $G_N$, we describe in Proposition \ref{Coset Decomposition} the double coset $P_N\backslash G_N/G_{n}\times G_{n}$. Then a Siegel-type Eisenstein series over $G_N$ can be decomposed into several orbits and except for one `main orbit' all orbits vanish when considering an inner product (in one variable) with the cusp form $\mathbf{f}$. This allows us to prove the following formula, see section 4 and in particular Theorem \ref{Integral Representation} for details and further notation,
\[
\int_{G_{n}(\Q)\backslash G_{n}(\mathbb{A})/K_1(\mathfrak{n})K_{\infty}}\mathbf{E}(g\times h,s)\mathbf{f}(h)\mathbf{d}h=c_k(s)D(s,\mathbf{f},\chi)\mathbf{f}(g),
\]
where here $\chi$ is a Dirichlet character, $c_k(s)$ is an explicit function on $s$, $D(s,\mathbf{f},\chi)$ is a Dirichlet series which is related (see Equation (\ref{Relation of Dirichlet series})) to the twisted standard $L$-function $L(s, \mathbf{f},\chi)$. \newline 

In section 5, we review the definition of algebraic modular forms and differential operators. It is well known how to define algebraic modular forms on hermitian symmetric space. There are several different definitions and we will mainly follow the one via CM point as in \cite{Sh00}. Using the Maass-Shimura differential operators we discuss the notion of a nearly holomorphic modular form in our setting. These differential operators for all four types of symmetric spaces mentioned above have been studied in \cite{Sh84,Sh94}. We will summarise the result there and apply it to the Siegel-type Eisenstein series mentioned above. Based on this and thanks to the well-understood Fourier expansion of Siegel-type Eisenstein series, we will prove our main algebraic result for L-functions by the same method as in \cite{Sh00}. Our main result is Theorem \ref{Main Theorem} which gives,

\begin{thm}
Let $\mathbf{f}\in\mathcal{S}_k(K_1(\mathfrak{n}),\overline{\Q})$ be an algebraic cuspidal Hecke eigenform, and let $\chi$ be a Dirichlet character whose conductor divides the ideal $\mathfrak{n}$. Assume $k>2n-1$ and let $\mu\in\Z$ such that $2n-1< \mu\leq k$. Then 
\[
\frac{L(\mu,\mathbf{f},\chi)}{\pi^{n(k+\mu)-\frac{3}{2}n(n-1)}\langle\mathbf{f,f}\rangle}\in\overline{\Q}.
\]
\end{thm}

\begin{rem} \label{Remark on the conductor}
We note here that the condition on the conductor of the Dirichlet character is not restrictive. Indeed, for $\mathbf{f}\in\mathcal{S}_k(K_1(\mathfrak{n}),\overline{\Q})$ and $\chi$ of conductor $\mathfrak{m}$ we can select $\mathfrak{n}'=\mathfrak{nm}$ instead of $\mathfrak{n}$ since $\mathcal{S}_k(K_1(\mathfrak{n}), \overline{\mathbb{Q}}) \subset \mathcal{S}_k(K_1(\mathfrak{n}'), \overline{\mathbb{Q}})$.
\end{rem}

Most of our arguments to prove the above are straightforward generalisation of \cite{Sh00} from the unitary and symplectic setting to our setting. Our main contribution is making some of the not always obvious generalisations as explicit as possible, such as the diagonal embedding, especially in the non-tube case (see section 2.3), the coset decomposition $P_N\backslash G_N/G_{n}\times G_{n}$ (see section 4)  and the definition of an algebraic modular form via values at CM points (see section 5).
Finally we should add that our computations are mainly using the adelic language (in comparison to the more classical in \cite{Sh00}), which is also inspired by \cite{P}.

\section{Groups and Symmetric Spaces}

In this and the next section, we introduce the notion of a quaternionic modular form and discuss some main properties. Such modular forms have been already studied (see for example \cite{B} and \cite{K}) but we extend the discussion to include also the case of non-split groups. For most of our notation here we follow the one introduced in the books \cite{Sh97} and \cite{Sh00}, where the case of Siegel and Hermitian modular forms is considered.

\subsection{Quaternionic Unitary Groups}

We start by fixing some notation. For more details on quaternion algebras the reader is refered to \cite{JV}. In this work a quaternion algebra will mean a central simple algebra of dimension four over $\Q$. After selecting a basis, we can write it in the form
\[
\mathbb{B}=\Q\oplus\Q\zeta\oplus\Q\xi\oplus\Q\zeta\xi,
\]
where
\[
\zeta^2=\alpha,\xi^2=\beta,\zeta\xi=-\xi\zeta.
\]
with $\alpha,\beta$ nonzero squarefree integers. We assume in this paper that $\mathbb{B}$ is definite, i.e. $\alpha,\beta<0$. The main involution of $\mathbb{B}$ is given by
\[
\overline{\cdot}:\mathbb{B}\to\mathbb{B}:a+b\zeta+c\xi+d\zeta\xi\mapsto\overline{a+b\zeta+c\xi+d\zeta\xi}=a-b\zeta-c\xi-d\zeta\xi.
\]
We warn the reader that we may, by abusing the notation, denote $\overline{\cdot}$ various involution of algebras (for example complex conjugation on quadratic imaginary fields), but it will be always clear from the context what is meant. The trace and norm are defined by $\mathrm{tr}(x)=x+\overline{x},N(x)=x\overline{x}$ for $x\in\mathbb{B}$. As usual we write $M_n(\mathbb{B})$ for the set of $n\times n$ matrices with entries in $\mathbb{B}$. We also use the notation $\mathbb{B}_n^m$ for the set of $m\times n$ matrices with entries in $\mathbb{B}$. For $X\in M_n(\mathbb{B})$, we write $X^{\ast}=\transpose{\overline{X}},\hat{X}=(X^{\ast})^{-1}$ for the conjugate transpose and its inverse (if makes sense). 

Identify $\zeta,\xi$ with $\sqrt{\alpha},\sqrt{\beta}\in\overline{\Q}$ and let $\mathbb{K}=\Q(\xi)$. We define the embedding
\[
\mathfrak{i}:\mathbb{B}\to M_2(\mathbb{K}),a+b\zeta+c\xi+d\zeta\xi\mapsto\left[\begin{array}{cc}
a+c\xi & \alpha(b-d\xi)\\
b+d\xi & a-c\xi
\end{array}\right].
\]
One easily checks that for $x\in\mathbb{B}$
\[
\mathfrak{i}(x)^{\ast}=I^{-1}\mathfrak{i}(x^{\ast})I,\,\,\,\,I:=\left[\begin{array}{cc}
-\alpha & 0\\
0 & 1
\end{array}\right],
\]
\[
\transpose{\mathfrak{i}(x)}=J^{-1}\mathfrak{i}(x^{\ast})J,\,\,\,\,\,J:=\left[\begin{array}{cc}
0 & -1\\
1 & 0
\end{array}\right],
\]
and $\mathfrak{i}$ induces an isomorphism
\[
\mathfrak{i}:\mathbb{B}\stackrel{\sim}\longrightarrow\{x\in M_2(\mathbb{K}):\overline{x}IJ=IJx\}.
\]
We extend this map to an embedding $\mathfrak{i}:M_n(\mathbb{B})\to M_{2n}(\mathbb{K})$ by sending $x=(x_{ij})$ to $(\mathfrak{i}(x_{ij}))$. Denote $I_n'=\mathrm{diag}[I,...,I],J_n'=\mathrm{diag}[J,...,J]$ with $n$ copies. Then for $x\in M_n(\mathbb{B})$
\[
\mathfrak{i}(x)^{\ast}=I_n^{\prime-1}\mathfrak{i}(x^{\ast})I_n',\,\,\,\,\,\transpose{\mathfrak{i}(x)}=J_n^{\prime-1}\mathfrak{i}(x^{\ast})J_n',
\]
and $\mathfrak{i}$ induces an isomorphism
\[
\mathfrak{i}:M_n(\mathbb{B})\stackrel{\sim}\longrightarrow\{x\in M_{2n}(\mathbb{K}):\overline{x}I_n'J_n'=I_n'J_n'x\}.
\] 

For a matrix with entries in quaternion algebra, the determinant $\det$ and trace $\mathrm{tr}$ will mean the reduced norm and reduced trace. That is taking the determinant and trace for its image under $\mathfrak{i}$. It is well known that the definition of reduced norm and trace is indeed independent of the choice of such embedding and the field $\mathbb{K}$. Denote
\[
\mathrm{GL}_n(\mathbb{B})=\{g\in M_n(\mathbb{B}):\det(g)\neq 0\},\mathrm{SL}_n(\mathbb{B})=\{g\in M_n(\mathbb{B}):\det(g)=1\}.
\]

Let $\mathbb{A}$ be the adele ring of $\Q$. By a place $v$, we mean either a finite place corresponding to a prime or the archimedean place $\infty$. The set of finite places is denoted as $\mathbf{h}$. We write $\mathbb{A}=\mathbb{A}_{\mathbf{h}}\R$ with finite adeles $\mathbb{A}_{\mathbf{h}}$ and $x=x_{\mathbf{h}}x_{\infty}$ with $x\in\mathbb{A},x_{\mathbf{h}}\in\mathbb{A}_{\mathbf{h}},x_{\infty}\in\R$. Fix embeddings $\overline{\Q}\to\overline{\Q}_v$ and set $\mathbb{B}_v=\mathbb{B}\otimes_{\Q}\Q_v$. The previous definition of trace, norm, determinant naturally extends locally or adelically. Fix a maximal order $\mathcal{O}$ of $\mathbb{B}$ and set $\mathcal{O}_v=\mathcal{O}\otimes_{\Z}\Z_v$. For a place $v$ we say $v$ splits if $\mathbb{B}_v\cong M_2(\Q_v)$. If this is the case we fix an isomorphism $\mathfrak{i}_v:\mathbb{B}_v\stackrel{\sim}\longrightarrow M_2(\Q_v)$ and assume $\mathfrak{i}_v(\mathcal{O}_v)=M_2(\Z_v)$ for finite place. $\mathbb{B}$ is called indefinite if $B_v$ is split for $v=\infty$ and definite otherwise. In particular, $\mathbb{B}$ is definite if $\beta<0$ and indefinite otherwise. That is, for the infinite place, by our assumption, $\mathbb{B}_{\infty}$ is the Hamilton quaternion
\[
\mathbb{H}=\R\oplus\R\mathbf{i}\oplus\R\mathbf{j}\oplus\R\mathbf{ij},\,\,\,\,\,\mathbf{i}^2=\mathbf{j}^2=-1,\mathbf{ij}=-\mathbf{ji},
\]
and the map $\mathfrak{i}$ above induces an isomorphism
\[
\mathfrak{i}:M_n(\mathbb{H})\stackrel{\sim}\longrightarrow \{x\in M_{2n}(\C):\overline{x}J_n'=J_n'x\}.
\]

In this paper, we consider following algebraic groups
\[
G=G_n(\Q)=\{g\in\mathrm{SL}_n(\mathbb{B}):g^{\ast}\phi g=\phi\},\phi=\left[\begin{array}{ccc}
0 & 0 & -1_m\\
0 & \zeta\cdot1_r & 0\\
1_m & 0 & 0
\end{array}\right].
\]
Here $n=2m+r$ and we assume $m>1$. Such a group is usually called a quaternionic unitary group. For a split finite place $v$, we make the technical assumption that $\zeta \in\Q_v^{\times 2}$ for all split places. Under this assumption, for every finite spilt place $v$, we have
\[
G(\Q_v)\cong\left\{g\in\mathrm{SL}_{2n}(\Q_v):\transpose{g}\left[\begin{array}{cc}
0 & 1_n\\
1_n &0
\end{array}\right]g=\left[\begin{array}{cc}
0 & 1_n\\
1_n &0
\end{array}\right]\right\}=:\mathrm{SO}(n,n)(\Q_v).
\]
We will discuss the local archimedean group $G(\R)$ and symmetric space in next subsection.

We fix an integral two-sided ideal $\mathfrak{n}=(\mathcal{N})$ of $\mathcal{O}$ generated by $\mathcal{N}=\prod_vp_v^{n_v}\in\Z$. We define an open compact subgroup $K_1(\mathfrak{n})\subset G(\mathbb{A}_{\mathbf{h}})$ by $K_1(\mathfrak{n})=\prod_vK_v$ where
\[
K_v=\left\{\gamma=\left[\begin{array}{ccc}
a & b & c\\
g & e & f\\
h & l & d
\end{array}\right]\in G(\mathcal{O}_v):\gamma\equiv\left[\begin{array}{ccc}
1_m & \ast & \ast\\
0 & 1_r & \ast\\
0 & 0 & 1_m
\end{array}\right]\text{ mod }p_v^{n_v}\right\}.
\]

It is well known, see for example \cite[page 251]{PR}, that we have a finite decomposition 
\[
G(\mathbb{A})=\bigcup_{j}G(\Q)t_jK_1(\mathfrak{n})G(\R).
\]
Moreover, thanks to the weak approximation which is valid for our group (see \cite[Proposition 7.11]{PR}), we can take $t_j$ such that $(t_j)_v=1$ for $v|\mathfrak{n}$ (compare with \cite[Lemma 8.12]{Sh97}). For finite places $v$ not in the support of $\mathfrak{n}$ the Iwasawa decomposition is valid and hence we can take $t_j$ to be upper triangular. Let $\Gamma_1^j=t_jK_1(\mathfrak{n})t_j^{-1}\cap G(\Q)$. We can take $t_0=1$ so that
\[
\Gamma_1^0=\Gamma_1(\mathcal{N})=\left\{\gamma=\left[\begin{array}{ccc}
a & b & c\\
g & e & f\\
h & l & d
\end{array}\right]\in G(\mathcal{O}):\gamma\equiv\left[\begin{array}{ccc}
1_m & \ast & \ast\\
0 & 1_r & \ast\\
0 & 0 & 1_m
\end{array}\right]\text{ mod } \mathcal{N}\right\}.
\]

\subsection{Symmetric spaces}

We start with a rather general setting. Let $\mathfrak{i}$ be any embedding $M_n(\mathbb{H})\to M_{2n}(\C)$. Then by the Skolem-Noether theorem there exists $\alpha\in M_{2n}(\C)$ with $\alpha\alpha^{\ast}=1$ such that $\transpose{\mathfrak{i}}(x)=\alpha\mathfrak{i}(x^{\ast})\alpha^{-1}$. Let $\Phi\in\GL_n(\mathbb{B})$ be a skew-hermitian form similar to $\phi$ above, that is $\Phi = \gamma^* \phi \gamma$ for some $\gamma \in GL_n(\mathbb{B})$. Then the group $G(\R)$ is isomorphic to
\[
\mathcal{G}=\{g\in\GL_{2n}(\C):g^{\ast}Hg=H,\transpose{g}Kg=K\},
\]
with $H=\mathfrak{i}(\Phi),K=\alpha^{-1}\mathfrak{i}(\Phi)$. We call it a realization of $G(\R)$. Suppose we are given two such data $(\mathfrak{i}_1,\Phi_1,H_1,K_1,\mathcal{G}_1)$ and $(\mathfrak{i}_2,\Phi_2,H_2,K_2,\mathcal{G}_2)$ with $\Phi_1=S^{\ast}\Phi_2 S$. Again by Skolem-Noether there exists $\beta$ with $\beta\beta^{\ast}=1$ such that $\mathfrak{i}_1(x)=\beta^{-1}\mathfrak{i}_2(x)\beta$. Put $R=\mathfrak{i}_2(S)\beta$ then $H_1=R^{\ast}H_2R,K_1=\transpose{R}K_2R$. Therefore $g\mapsto RgR^{-1}$ gives isomorphism $\mathcal{G}_1\cong\mathcal{G}_2$. 

Following \cite{PS}, we will define the associated symmetric space via its Borel embedding into its compact dual symmetric space. In our case we have that the semisimple compact dual of our group is the group $\mathrm{SO}(2n)$ (see \cite[page 330]{Helgason}), and the corresponding dual symmetric space is $\mathrm{SO}(2n)/\mathrm{U}(n)$. This space may be identified (see for example \cite[page 6]{Sh87}) with the space  $V = L/ GL_{n}(\mathbb{C})$ where 
\[
L = \{ U \in\C_n^{2n}: \,\,\, \transpose{U}KU=0\}.
\]
We set
\[
\Omega=\{U\in\C_{2n}^{n}:-iU^{\ast}HU>0,\transpose{U}KU=0\}\subset L,
\]
with the action of $\GL_n(\C)$ by right multiplication and $\mathcal{G}$ by left multiplication. The symmetric space $\mathcal{H}$ is defined as
\[
\mathcal{H}=\{z\in\C_n^n:U(z)\in\Omega\}, \,\,\,\,U(z):=\left[\begin{array}{c}
z\\
u_0
\end{array}\right],
\]
for some fixed suitable $u_0$, which we make explicit later. The following lemma is a direct consequence of our definition for $\mathcal{H}$.

\begin{lem}
There is a bijection $\mathcal{H}\times\GL_n(\C)\to\Omega$ given by $z\times\lambda=U(z)\lambda$.
\end{lem}

Note that $\mathcal{G}$ acts on $\Omega$ by left multiplication. By the above lemma, it follows that for any element $\alpha\in \mathcal{G}$, we can find a $z'\in\mathcal{H}$ and an $\lambda(\alpha,z)\in\GL_n(\C)$ such that
\[
\alpha U(z)=U(z')\lambda(\alpha,z).
\]
We then define the action of $G(\R)$ on $\mathcal{H}$ by $\alpha.z:=\alpha z:=z'$ and $\lambda(\alpha,z)$  satisfies the cocycle relation
\[
\lambda(\alpha_1\alpha_2,z)=\lambda(\alpha_1,\alpha_2 z)\lambda(\alpha_2,z)\text{ for }\alpha_1,\alpha_2\in \mathcal{G},z\in\mathcal{H}.
\]
We set $j(\alpha,z):=\det(\lambda(\alpha,z))\in\C^{\times}$. We call $\lambda(\alpha,z)$ or $j(\alpha,z)$ automorphy factors. More explicitly, write $\alpha=\left[\begin{array}{cc}
a & b\\
c & d
\end{array}\right]$,
\[
\alpha U(z)=\left[\begin{array}{c}
az+bu_0\\
cz+du_0
\end{array}\right]=\left[\begin{array}{c}
(az+bu_0)(cz+du_0)^{-1}u_0\\
u_0
\end{array}\right]u_0^{-1}(cz+du_0).
\]
That is, $\alpha z=(az+bu_0)(cz+du_0)^{-1}u_0$,and $\lambda(\alpha,z)=u_0^{-1}(cz+du_0)$.

For $z_1,z_2\in\mathcal{H}$, we set 
\[
\eta(z_1,z_2):=iU(z_1)^{\ast}HU(z_2), \,\,\delta(z_1,z_2):=\det(\eta(z_1,z_2))\,\, \text{and} \,\,\eta(z):=\eta(z,z),\delta(z):=\delta(z,z).
\]

We now note that
\[
U(z_1)^{\ast}HU(z_2)=\lambda(\alpha,z_1)^{\ast}U(\alpha z_1)^{\ast}HU(\alpha z_2)\lambda(\alpha,z_2),
\]
and
\[
iU(\alpha z_1)^{\ast}HU(\alpha z_2)=\left[\begin{array}{cc}\eta(\alpha z_1,\alpha z_2) & \ast\\
\ast & \ast
\end{array}\right],iU(z_1)^{\ast}HU( z_2)=\left[\begin{array}{cc}\eta(z_1,z_2) & \ast\\
\ast & \ast
\end{array}\right].
\]
In particular we obtain that
\[
\lambda(\alpha,z_1)^{\ast}\eta(\alpha z_1,\alpha z_2)\lambda(\alpha,z_2)=\eta(z_1,z_2),
\]
and after taking the determinant, we have
\[
\overline{j(\alpha,z_1)}\delta(\alpha z_1,\alpha z_2)j(\alpha,z_2)=\delta(z_1,z_2).
\]
In particular,
\[
\lambda(\alpha,z)^{\ast}\eta(\alpha z)\lambda(\alpha,z)=\eta(z),\,\,\,\delta(\alpha z)=|j(\alpha,z)|^{-2}\delta(z).
\]

We now discuss the relation between different realizations of the symmetric space $\mathcal{H}$. Given $H_1,K_1$ and $H_2,K_2$ as above, we have seen at the beginning of this subsection that we can find an $R$ such that $H_1=R^{\ast}H_2R,K_1=\transpose{R}K_2R$. We then have an isomorphism $\Omega_1\cong\Omega_2$ given by $U\mapsto RU$ which induces isomorphism $\rho:\mathcal{H}_1\cong\mathcal{H}_2$. Indeed, for $z_1\in\mathcal{H}_1$, there exists some $z_2\in\mathcal{H}_2,\mu(z_1)\in\GL_n(\C)$ such that 
\begin{equation} \label{The matrix R}
R\left[\begin{array}{c}z_1\\u_{01}\end{array}\right]=\left[\begin{array}{c}z_2\\u_{02}\end{array}\right]\mu(z_1),
\end{equation}
and the isomorphism can be given by $\rho(z_1)=z_2$.

In the following lemma we write $\rho$ also for the isomorphism $\mathcal{G}_1 \rightarrow \mathcal{G}_2$ given by $\rho(g_1):= R g_1 R^{-1}$.

\begin{lem} \label{comparison}
Let $\rho:\mathcal{G}_1\to \mathcal{G}_2,\rho:\mathcal{H}_1\to\mathcal{H}_2$ given as above. Then\\
(1) $\rho(\alpha z)=\rho(\alpha)\rho(z)$ with $\alpha\in \mathcal{G}_1,z\in\mathcal{H}_1$;\\
(2) $\lambda(\rho(\alpha),\rho(z))=\mu(\alpha z)\lambda(\alpha,z)\mu(z)^{-1}$;\\
(3) $\eta(\rho(z_1),\rho(z_2))=\widehat{\mu(z_1)}\eta(z_1,z_2)\mu(z_2)^{-1}$ for $z_1,z_2\in\mathcal{H}_1$.
\end{lem}

\begin{proof}
(1) It suffices to prove that $\left[\begin{array}{c}\rho(\alpha z)\\u_{02}\end{array}\right]=\left[\begin{array}{c}\rho(\alpha)\rho(z)\\u_{02}\end{array}\right]$. By definition of the isomorphism and action, 
\[
\left[\begin{array}{c}\rho(\alpha z)\\u_{02}\end{array}\right]=R\left[\begin{array}{c}
\alpha z\\
u_{01}
\end{array}\right]\mu(\alpha z)^{-1}=
R\alpha\left[\begin{array}{c}
z\\
u_{01}
\end{array}\right]\lambda(\alpha,z)^{-1}\mu(\alpha z)^{-1}=
\]
\[
\rho(\alpha)\left[\begin{array}{c}
\rho(z)\\
u_{02}
\end{array}\right]\mu(z)\lambda(\alpha,z)^{-1}\mu(\alpha z)^{-1}=\left[\begin{array}{c}
\rho(\alpha)\rho(z)\\
u_{02}
\end{array}\right]\lambda(\rho(\alpha),\rho(z))\mu(z)\lambda(\alpha,z)^{-1}\mu(\alpha z)^{-1}.
\]
We must have $\lambda(\rho(\alpha),\rho(z))\mu(z)\lambda(\alpha,z)^{-1}\mu(\alpha z)^{-1}=1$ and our desired result follows which we also obtain (2). (3) can be computed similarly by definition of $\eta$.
\end{proof}

We now apply the above considerations to some explicit realisations of $G(\R)$. Note that the map $\mathfrak{i}$ defined above induces the following isomorphism on $\Q$-groups
\[
\mathfrak{i}:G\stackrel{\sim}\longrightarrow\mathfrak{G}=\{g\in\GL_{2n}(\mathbb{K}):g^{\ast}\Phi g=\Phi,\transpose{g}\Psi g=\Psi\},
\]
\[
\Phi=\left[\begin{array}{cccc}
0 & 0 & 0 & -1_{2m}\\
0 & 0 & -1_r & 0\\
0 & 1_r & 0 & 0\\
1_{2m} & 0 & 0 & 0
\end{array}\right],\Psi=\left[\begin{array}{cccc}
0 & 0 & 0 & J_m'I_m'\\
0 & -\alpha^{-1} & 0 & 0\\
0 & 0 & 1_r & 0\\
-J_m'I_m' & 0 & 0 & 0
\end{array}\right].
\]
This induces following isomorphism on $\R$-groups
\[
\mathfrak{i}:G(\R)\stackrel{\sim}\longrightarrow G_{\infty}:=\{g\in\GL_{2n}(\C):g^{\ast}\phi_{\infty}g=\phi_{\infty},\transpose{g}\psi_{\infty}g=\psi_{\infty}\},
\] 
\[
\phi_{\infty}=\left[\begin{array}{cccc}
0 & 0 & 0 & -1_{2m}\\
0 & 0 & -1_r & 0\\
0 & 1_r & 0 & 0\\
1_{2m} & 0 & 0 & 0
\end{array}\right],\psi_{\infty}=\left[\begin{array}{cccc}
0 & 0 & 0 & J_m'\\
0 & 1_r & 0 & 0\\
0 & 0 & 1_r & 0\\
-J_m' & 0 & 0 & 0
\end{array}\right].
\]
Let $K_{\infty}$ be a maximal compact subgroup of $G(\R)$. As in last subsection
\[
\Omega=\{U\in\C_n^{2n}:-iU^{\ast}\phi_{\infty}U>0,\transpose{U}\psi_{\infty}U=0\},
\]
and define the symmetric space by
\[
\mathfrak{Z}=\mathfrak{Z}_n=\mathfrak{Z}_{m,r}=\{z\in\C_n^n:U(z)\in\Omega\},U(z)=\left[\begin{array}{c}
z\\
u_0
\end{array}\right],u_0=\left[\begin{array}{cc}
0 & 1_r\\
1_{2m} & 0
\end{array}\right].
\]
Explicitly,
\[
\mathfrak{Z}=\left\{z=(u,v,w):=\left[\begin{array}{cc}
u & v\\
w\transpose{v}J_m' & w
\end{array}\right]:\begin{array}{c}
u\in\C_{2m}^{2m},v\in\C_r^{2m},w\in\C_r^r,i(z^{\ast}-z)>0,
\\
\transpose{w}w+1=0,uJ_m'+v\transpose{v}-J_m'\transpose{u}=0.
\end{array}\right\}.
\]
The action of $G_{\infty}$ on $\mathfrak{Z}$ is given by 
\[
g z=(az+bu_0)(cz+du_0)^{-1}u_0,\lambda(g,z)=u_0^{-1}(cz+du_0),g=\left[\begin{array}{cc}
a & b\\
c & d
\end{array}\right]\in G_{\infty}.
\]

For $z_1,z_2\in\mathfrak{Z}$, we set $\eta(z_1,z_2)=i(z_1^{\ast}-z_2),\delta(z_1,z_2)=\det(\eta(z_1,z_2))$ and $\eta(z)=\eta(z,z),\delta(z)=\delta(z,z)$. We will take $z_0=i\cdot 1_n$ to be the origin of $\mathfrak{Z}$ and $K_{\infty}$ the subgroup of $G_{\infty}$ fixing $z_0$. Then $g\mapsto\lambda(g,z_0)$ gives an isomorphsim $K_{\infty}\cong U(n)=\{g\in\GL_n(\C):g^{\ast}g=1_n\}$ and our symmetric space $\mathfrak{Z}\cong G_{\infty}/K_{\infty}$. We note that we are using the same notation $K_{\infty}$ for maximal subgroup of $G_{\infty}$ and its preimage in $G(\R)$.

We now give another two useful realizations. Note that $G$ is isomorphic to 
\[
G'=\{g\in\GL_n(\mathbb{B}):g^{\ast}\phi'g=\phi'\},\phi'=\left[\begin{array}{ccc}
\zeta\cdot 1_m & 0 & 0\\
0 & \zeta\cdot 1_r & 0\\
0 & 0 & -\zeta\cdot 1_m
\end{array}\right].
\]
By changing rows and colomns the map $\mathfrak{i}$ induces isomorphism
\[
\mathfrak{i}:G'\stackrel{\sim}\longrightarrow\mathfrak{G}'=\{g\in\GL_{2n}(\mathbb{K}):g^{\ast}\Phi' g=\Phi',\transpose{g}\Psi' g=\Psi'\},
\]
\[
\Phi'=J_n,\Psi'=\mathrm{diag}[1,1,1,-\alpha,-\alpha,-\alpha],
\]
and
\[
\mathfrak{i}:G'(\R)\cong G_{\infty}':=\{g\in\GL_{2n}(\C):g^{\ast}J_ng=J_n,\transpose{g}g=1_{2n}\}.
\]
Take $u_0=1$, the symmetric space associated to this group is
\[
\mathfrak{H}=\mathfrak{H}_n=\{z\in\C_n^n:\transpose{z}z+1=0,i(z^{\ast}-z)>0\}.
\]
This is an unbounded realization of type D domain in \cite{LKW}. The action of $G_{\infty}'$ on $\mathfrak{H}$ and the automorphy factor is given by
\[
gz=(az+b)(cz+d)^{-1},\lambda(g,z)=cz+d,g=\left[\begin{array}{cc}
a & b\\
c & d
\end{array}\right].
\]
For $z_1,z_2\in\mathfrak{H}$, we set $\eta(z_1,z_2)=i(z_1^{\ast}-z_2)$. We take $z_0=i_n$ to be the origin of $\mathfrak{H}$ and $K_{\infty}'$ the subgroup of $G_{\infty}'$ fixing $z_0$. Since $\eta(gz_0)=\eta(z_0)=2$ for $g\in K_{\infty}$, $g\mapsto\lambda(g,z_0)$ gives an isomorphism $K_{\infty}'\cong U(n)$ and thus $\mathfrak{H}\cong G_{\infty}'/K_{\infty}'$.

Let $T'=\frac{1}{\sqrt{2}}\left[\begin{array}{cc}
i & -i\\
1 & 1
\end{array}\right]$ and sending $g\mapsto T^{\prime-1}gT'$ we have isomorphism
\[
\mathfrak{i}_{\infty}'':G_{\infty}'\stackrel{\sim}\longrightarrow G_{\infty}''=\{g\in\GL_{2n}(\C):g^{\ast}\phi_{\infty}''g=\phi_{\infty}'',\transpose{g}\psi_{\infty}''g=\psi_{\infty}''\},
\]
with
\[
\phi_{\infty}''=\left[\begin{array}{cc}
i_n & 0\\
0 & -i_n
\end{array}\right],\psi_{\infty}''=\left[\begin{array}{cc}
0 & -i_n\\
-i_n & 0
\end{array}\right].
\]
Take $u_0=1$ the symmetric space associated to this group is defined as
\[
\mathfrak{B}=\mathfrak{B}_n=\{z\in\C_n^n:\transpose{z}=-z,zz^{\ast}<1_n\}.
\]
This is a bounded domain of type $\mathfrak{R}_{\mathrm{III}}$ in \cite{Hua}. The action of $G_{\infty}''$ on $\mathfrak{B}$ and the automorphy factor is given by
\[
gz=(az+b)(cz+d)^{-1},\lambda(g,z)=cz+d,g=\left[\begin{array}{cc}
a & b\\
c & d
\end{array}\right].
\]
For $z_1,z_2\in\mathfrak{H}$, we set $\eta(z_1,z_2)=i(z_1^{\ast}z_2-1)$. We take $z_0=0$ to be the origin of $\mathfrak{B}$ and $K_{\infty}''$ the subgroup of $G_{\infty}''$ fixing $z_0$. Since $\eta(gz_0)=\eta(z_0)=-i$ for $g\in K_{\infty}$, $g\mapsto\lambda(g,z_0)$ gives an isomorphism $K_{\infty}''\cong U(n)$ and thus $\mathfrak{H}\cong G_{\infty}''/K_{\infty}''$. The relation between $\mathfrak{H}$ and $\mathfrak{B}$ can be given explicitly by Cayley transform
\[
\mathfrak{H}\stackrel{\sim}\longrightarrow\mathfrak{B}:z\mapsto(z-i)(z+i)^{-1}.
\]

Let $z_1,z_2\in\mathfrak{B}_n,\alpha\in G(\R)$ as above and $dz=(dz_{hk})$ be a matrix of the same shape as $z\in\C_n^n$ whose entries are $1$-forms $dz_{hk}$. Comparing
\[
\left[\begin{array}{cc}
z_1 & 1\\
1 & -\overline{z}_1
\end{array}\right]^{\ast}\left[\begin{array}{cc}
1_n & 0\\
0 & -1_n
\end{array}\right]\left[\begin{array}{cc}
z_2 & 1\\
1 & -\overline{z}_2
\end{array}\right]=\left[\begin{array}{cc}
z_1^{\ast}z_2-1 & z_1^{\ast}+\overline{z}_2\\ 
z_2+\transpose{z}_1 & 1-\transpose{z}_1\overline{z}_2
\end{array}\right]=\left[\begin{array}{cc}z_1^{\ast}z_2-1 & z_1^{\ast}-z_2^{\ast}\\ 
z_2-z_1 & 1-\transpose{z}_1\overline{z}_2\end{array}\right],
\]
\[
\left[\begin{array}{cc}
\alpha z_1 & 1\\
1 & -\overline{\alpha z}_1
\end{array}\right]^{\ast}\left[\begin{array}{cc}
1_n & 0\\
0 & -1_n
\end{array}\right]\left[\begin{array}{cc}
\alpha z_2 & 1\\
1 & -\overline{\alpha z}_2
\end{array}\right]=\left[\begin{array}{cc}(\alpha z_1)^{\ast}(\alpha z_2)-1 & (\alpha z_1)^{\ast}-(\alpha z_2)^{\ast}\\ 
\alpha z_2-\alpha z_1 & 1-\transpose{(\alpha z}_1)(\overline{\alpha z}_2)\end{array}\right],
\]
and using the fact (which can be obtained from the property of $U(z)$)
\[
\alpha\left[\begin{array}{cc}
z & 1\\
1 & -\overline{z}
\end{array}\right]=\left[\begin{array}{cc}
\alpha z & 1\\
1 & -\overline{\alpha z}
\end{array}\right]\left[\begin{array}{cc}
\lambda(\alpha,z) & 0\\
0 &\overline{\lambda(\alpha,z)} 
\end{array}\right],
\]
we have
\[
\alpha z_2-\alpha z_1=\transpose{\lambda}(\alpha,z_1)^{-1}(z_2-z_1)\lambda(\alpha,z_2)^{-1}.
\]
Therefore, 
\[
d(\alpha z)=\transpose{\lambda(\alpha,z)}^{-1}\cdot dz\cdot \lambda(\alpha,z)^{-1}.
\]
Since the jacobian of the map $z\mapsto\alpha z$ is $j(\alpha,z)^{-n+1}$, the differential form
\[
\mathbf{d}z=\delta(z)^{-n+1}\prod_{h\leq k}[(i/2)dz_{hk}\wedge d\overline{z}_{hk}],
\]
is an invariant measure. If we have another realizaton $\mathcal{H}$ (e.g. $\mathfrak{Z,H}$) with identification $\rho:\mathcal{H}\to\mathcal{B}$, we then define $\mathbf{d}z:=\mathbf{d}(\rho(z))$ with $z\in\mathcal{H}$ to be the differential form on $\mathcal{H}$. Clearly, this is also an invariant measure.

%\newpage
\subsection{Doubling Embedding}

We keep the notation as before and consider two groups 
\[
G_{n_1}=\left\{g\in\GL_{n_1}(\mathbb{B}):g^{\ast}\phi_1 g=\phi_1\right\},\phi_1=\left[\begin{array}{ccc}
0 & 0 & -1_{m_1}\\
0 & \zeta\cdot 1_r & 0\\
1_{m_1} & 0 & 0
\end{array}\right],
\]
\[
G_{n_2}=\left\{g\in\GL_{n_2}(\mathbb{B}):g^{\ast}\phi_2 g=\phi_2\right\},\phi_2=\left[\begin{array}{ccc}
0 & 0 & -1_{m_2}\\
0 & \zeta\cdot 1_r & 0\\
1_{m_2} & 0 & 0
\end{array}\right],
\]
with $n_1=2m_1+r,n_2=2m_2+r$. We always assume $m_1\geq m_2>0$. We set $N=n_1+n_2 = 2m_1+2m_2+2r$, and consider the map
\[
G_{n_1}\times G_{n_2}\to G^{\omega}=\{g\in\GL_{N}(\mathbb{B}):g^{\ast}\omega g=\omega\},\omega=\left[\begin{array}{cc}
\phi_1 & 0\\
0 & -\phi_2
\end{array}\right],
\]
by sending $g_1\times g_2\mapsto\mathrm{diag}[g_1,g_2]$. Note that $R^{\ast}\omega R=J_{N/2} := \left[\begin{array}{cc}
 0 & -1_{N/2}\\
1_{N/2}  & 0
\end{array}\right]$,

 with
\[
R=\left[\begin{array}{cccccc}
1_{m_1} & 0 & 0 & 0 & 0 & 0\\
0 & 1/2 & 0 & 0 &-\zeta^{-1} & 0\\
0 & 0 & 0 & 1_{m_2} & 0 & 0\\
0 & 0 & -1_{m_1} & 0 & 0 & 0\\
0 & -1/2 & 0 & 0 & -\zeta^{-1}  & 0\\
0 & 0 & 0 & 0 & 0 & 1_{m_2}
\end{array}\right].
\]
Composing the above map with $g\mapsto R^{-1}gR$ we obtain an embedding
\[
\rho:G_{n_1}\times G_{n_2}\to G^{\omega}\to G_N=\{g\in\GL_{N}(\mathbb{B}):g^{\ast}J_{N/2} g=J_{N/2}\}.
\]
We can thus view $G_{n_1}\times G_{n_2}$ as a subgroup of $G_N$. To ease notation, we write $G_{n_1}\times G_{n_2}$ as its image in $G_N$ under $\rho$ and $\beta\times\gamma$ for $\beta\in G_{n_1},\gamma\in G_{n_2}$ as an element in $G_N$ under the embedding $\rho$. 
\begin{comment}
Let $n_1=2m_1+r,n_2=2m_2$, we define following embedding
\[
G_{n_1}\times G_{n_2}\to G_{n_1+n_2}=\left\{g\in\GL_{n_1+n_2}:g^{\ast}\phi'g=\phi'\right\},\phi'=\left[\begin{array}{ccc}
0 & 0 & -1_{m_1+m_2}\\
0 & \zeta\cdot 1_r & 0\\
1_{m_1+m_2} & 0 & 0
\end{array}\right]
\]
by sending $g_1\times g_2\mapsto\mathrm{diag}[g_1,g_2]\mapsto T^{-1}\mathrm{diag}[g_1,g_2]T$ with 
\[
T=\left[\begin{array}{ccccc}
1_{m_1} & 0 & 0 & 0 & 0\\
0 & 0 & 1_r & 0 & 0\\
0 & 0 & 0 & 1_{m_1} & 0\\
0 & -1_{m_2} & 0 & 0 & 0\\
0 & 0 & 0 & 0 & 1_{m_2}
\end{array}\right].
\]
Similar construction of embedding makes sense if $n_1=2m_1,n_2=2m_2+r$. In the sequel we will omit specifying an embedding if it is clear from the size of groups. We shall always write $G_{n_1}\times G_{n_2}$ as its image in $G_{n_1+n_2}$ and $\beta\times\gamma$ for $\beta\in G_{n_1},\gamma\in G_{n_2}$ as an element in $G_{n_1+n_2}$ under the embedding. 
\end{comment}

Now we consider a special case of this embedding, namely the case where $m_1=m_2=m,n_1=n_2=n$. To ease the notation we always omit the subscript `$n$' and keep the subscript $N=2n$. The embedding $\rho$ then induces
\[
\rho:G_{\infty}\times G_{\infty}\to G_{N\infty},g_1\times g_2\mapsto R^{-1}\mathrm{diag}[g_1,g_2]R,
\]
with
\[
R=R_{\infty}=\left[\begin{array}{cccccccc}
1_{2m} & 0 & 0 & 0 & 0 & 0 & 0 & 0\\
0 & 1/2 & 0 & 0 & 0 & 0 & -1 & 0\\
0 & 0 & 1/2 & 0 & 0 & 1 & 0 & 0\\
0 & 0 & 0 & 0 & 1_{2m} & 0 & 0 & 0\\
0 & 0 & 0 & -1_{2m} & 0 & 0 & 0 & 0\\
0 & -1/2 & 0 & 0 & 0 & 0 & -1 & 0\\
0 & 0 & -1/2 & 0 & 0 & 1 & 0 & 0\\
0 & 0 & 0 & 0 & 0 & 0 & 0 & 1_{2m}
\end{array}\right],
\]
where the entries with $\pm1,\pm1/2$ should be understood as $\pm I_r,\pm \frac{1}{2}I_r$.
Let $\mathfrak{Z},\mathfrak{Z}_N$ be symmetric spaces associated to $G_{\infty},G_{N\infty}$. We are now going to define an associated embedding of symmetric space $\iota:\mathfrak{Z}\times\mathfrak{Z}\to\mathfrak{Z}_N$. We first define the embedding 
\[
\Omega_1\times\Omega_2\to\Omega_N,U_1\times U_2\mapsto R^{-1}\left[\begin{array}{cc}
U_1 & 0\\
0 & \mathfrak{J}\overline{U_2}
\end{array}\right],\mathfrak{J}=\left[\begin{array}{cccc}
J_{m}' & 0 & 0 & 0\\
0 & 0 & -1_r & 0\\
0 & 1_r & 0 & 0\\
0 & 0 & 0 & J_{m}'
\end{array}\right].
\]
Let $z_1= (u_1,v_1,w_1),z_2=(u_2,v_2,w_2) \in\mathfrak{Z}$. The image of $U(z_1)\times U(z_2) \in \Omega_1 \times \Omega_2 $ under this map is
\[
\left[\begin{array}{cccc}
u_1 & v_1 & 0 & 0\\
w_1\transpose{v}_1J_{m}' & w_1 & 0 & 1\\
0 & 1 & -\overline{w}_2v_2^{\ast}J_{m}' & -\overline{w}_2\\
0 & 0 & -J_{m}'\overline{u}_2 & -J_{m}'\overline{v}_2\\
1 & 0 & 0 & 0\\
0 & \frac{1}{2} & \frac{\overline{w}_2v_2^{\ast}J_{m}'}{2} & \frac{\overline{w}_2}{2}\\
\frac{-w_1\transpose{v}_1J_{m}'}{2} & \frac{-w_1}{2} & 0 & \frac{1}{2}\\
0 & 0 & J_{m}' & 0
\end{array}\right]=:\left[\begin{array}{c}
A(z_1,z_2)\\
B(z_1,z_2)
\end{array}\right].
\]
We then define the embedding of symmetric space as
\[
\iota:\mathfrak{Z}\times\mathfrak{Z}\to\mathfrak{Z}_N,\,\,\,\,z_1\times z_2\mapsto A(z_1,z_2)B(z_1,z_2)^{-1}S,
\]
where $S:=\mathrm{diag}[1,1/2,1/2,1]$. \newline

Explicitly, if we write $w^{-1}_0:=1+w_1\overline{w}_2,w_0':=1-w_1\overline{w}_2$, then 
\[
\iota(z_1,z_2)=
\]
\[
\left[\begin{array}{cccc}
u_1-v_1\overline{w}_2w_0w_1\transpose{v}_1J_{m}' & v_1w_1^{-1}w_0w_1 & -v_1\overline{w}_2w_0 & -v_1w_1^{-1}w_0w_1\overline{w}_2v_2^{\ast}\\
2w_0w_1\transpose{v}_1J_{m}' & 2w_0w_1 & w_0'w_0 & -2w_0w_1\overline{w}_2v_2^{\ast}\\
-2\overline{w}_2w_0w_1\transpose{v}_1J_{m}' & w_1^{-1}w_0'w_0w_1 & -2\overline{w}_2w_0 & -2w_1^{-1}w_0w_1\overline{w}_2v_2^{\ast}\\
-J_{m}'\overline{v}_2w_0w_1\transpose{v}_1J_{m}' & -J_{m}'\overline{v}_2w_0w_1 & -J_{m}'\overline{v}_2w_0 & -J_{m}'\overline{u}_2J_{m}^{\prime-1}+J_{m}'\overline{v_2}w_0w_1\overline{w}_2v_2^{\ast}
\end{array}\right].
\]

Here we note that we have ``normalised'' our embedding by $S$  so that $\iota$ maps the origin of $\mathfrak{Z}\times\mathfrak{Z}$ to the ``origin'' of $\mathfrak{Z}_N$. That is, $\iota(z_0\times z_0)=i\cdot 1_N=:Z_0$, where $z_0$, and $Z_0$ are the origins of $\mathfrak{Z}$ and $\mathfrak{Z}_N$ respectively.%We may also write $z_1\times z_2$ for its image under $\iota$. 

% We use the same notation if its size is clear from the context.

For example in the case where $r=0$ then the embedding is quite simple, namely $\iota(z_1,z_2) = \diag[u_1,-u^*_2]$, where in the case of $r=1$ it is given by (note that in this case $w_1=w_2 =i$)
\[
\iota(z_1,z_2) = \left[\begin{array}{cccc}
u_1-\frac{1}{2}v_1\transpose{v}_1J_m' & \frac{1}{2}v_1 & -\frac{i}{2}v_1 & \frac{i}{2}v_1v_2^{\ast}\\
i\transpose{v}_1J_m' & i & 0 & -v_2^{\ast}\\
-\transpose{v}_1J_{m}' & 0 & i & iv_2^{\ast}\\
-\frac{i}{2}J_{m}'\overline{v}_2\transpose{v}_1J_{m}' & -\frac{i}{2}J_m'\overline{v}_2 & -\frac{1}{2}J_{m}'\overline{v}_2 & -u_2^{\ast}-\frac{1}{2}J_m'\overline{v}_2v_2^{\ast}
\end{array}\right].
\]

We now show that the embedding of the symmetric spaces is compatible with the embedding of the groups. 
\begin{prop} \label{diagonal embedding}
For $g_1,g_2\in G_{\infty},z_1,z_2\in\mathfrak{Z}$ we have\\
(1) $\iota(g_1z_1,g_2z_2)=\rho(g_1,g_2)\iota(z_1,z_2)$;\\
(2) $j(\rho(g_1,g_2),\iota(z_1,z_2))\det(B(z_1,z_2))=j(g_1,z_1)\overline{j(g_2,z_2)}\det(B(g_1z_1,g_2z_2))$;\\
(3) $\delta(\iota(z_1,z_2))=|\det(B(z_1,z_2))|^{-2}\delta(z_1)\delta(z_2)$.
\end{prop}

\begin{proof}
By our definition of embedding and the action
\[
\left[\begin{array}{c}
\iota(g_1z_2,g_2z_2)\\
1
\end{array}\right]=R^{-1}\left[\begin{array}{cc}
U(g_1z_1) & 0\\
0 & \mathfrak{J}\overline{U(g_2z_2)}
\end{array}\right]B(g_1z_1,g_2z_2)^{-1}S
\]
\[
=R^{-1}\left[\begin{array}{cc}
g_1 & 0\\
0 & g_2
\end{array}\right]RR^{-1}\left[\begin{array}{cc}
U(z_1) & 0\\
0 & \mathfrak{J}\overline{U(z_2)}
\end{array}\right]\left[\begin{array}{cc}
\lambda(g_1,z_1) & 0\\
0 & \overline{\lambda(g_2,z_2)}
\end{array}\right]^{-1}B(g_1z_1,g_2z_2)^{-1}S
\]
\[
=\left[\begin{array}{c}
\rho(g_1,g_2)\iota(z_1,z_2)\\
1
\end{array}\right]\lambda(\rho(g_1,g_2),\iota(z_1,z_2))\times
\]
\[
S^{-1}B(z_1,z_2)\left[\begin{array}{cc}
\lambda(g_1,z_1) & 0\\
0 & \overline{\lambda(g_2,z_2)}
\end{array}\right]^{-1}B(g_1z_1,g_2z_2)^{-1}S.
\]
We must have 
\[
\lambda(\rho(g_1,g_2),\iota(z_1,z_2))S^{-1}B(z_1,z_2)\left[\begin{array}{cc}
\lambda(g_1,z_1) & 0\\
0 & \overline{\lambda(g_2,z_2)}
\end{array}\right]^{-1}B(g_1z_1,g_2z_2)^{-1}S=1,
\]
and the desired result follows. Taking the determinant we also obtain (2).

Suppose $z_1=g_1z_0,z_2=g_2z_0$ for $g_1\in G_{1\infty},g_2\in G_{2\infty}$ with $z_0$ the origin of $\mathfrak{Z}_1$ or $\mathfrak{Z}_2$. Then
\[
\delta(\iota(z_1,z_2))=\delta(\rho(g_1,g_2)\iota(z_0,z_0))=|j(\rho(g_1,g_2),\iota(z_0,z_0))|^{-2}\delta(\iota(z_0,z_0))
\]
\[
=|j(g_1,z_0)j(g_2,z_0)\det(B(z_0,z_0)^{-1}B(g_1z_0,g_2z_0))|^{-2}\delta(z_0)\delta(z_0)
\]
\[
=|\det(B(z_1,z_2))|^{-2}\delta(z_1)\delta(z_2).
\]
\end{proof}

\section{Quaternionic Modular Forms, Hecke Operators and L-functions}

In this section we introduce the notion of a modular form (of scalar weight) and define the Hecke operators in our setting. We then define the associated (standard) $L$-function. We keep writing $G$ for $G_n$ with $n = 2m +r$ as above.

\subsection{Modular forms and Fourier-Jacobi expansion}

Fix an integral ideal $\mathfrak{n}$ as last section.

\begin{defn}
A holomorphic function $f:\mathfrak{Z}\to\C$ is called a quaternionic modular form for a congruence subgroup $\Gamma$ and weight $k \in \mathbb{N}$ if for all $\gamma\in\Gamma$,
\[
f(\gamma z)=j(\gamma,z)^kf(z).
\]
\end{defn}

We note here that since we are assuming $m \geq 2$ we do not need any condition at the cusps. 

Denote the space of such functions by $M_k(\Gamma)$. Here we are using the realization $(G_{\infty},\mathfrak{Z})$ for our symmetric space. In fact, the definition is independent of the choice of realizations in following sense. If we choose another realization $\mathcal{H}$ (e.g. $\mathfrak{H},\mathfrak{B}$) with identification $\rho:\mathcal{H}\to\mathfrak{Z}$. Then with notation as in Equation (\ref{The matrix R}), to a function $f:\mathfrak{Z}\to\C$ we associate a function $g$ on $\mathcal{H}$ by setting $g(z)=\det(\mu(z))^{-k}f(\rho(z))$. Then $f:\mathfrak{Z}\to\C$ is a modular form if and only if $g:\mathcal{H}\to\C$ is a modular form. 

We write $S:=S(\Q):=\{X\in M_m(\mathbb{B}):X^{\ast}=X\}$ for the (additive) algebraic group of hermitian matrices. We use $S^+$ (resp. $S_+$) denote the subgroup of $S$ consisting of positive definite (resp. positive) elements. For a fractional ideal $\mathfrak{a}\subset\mathbb{B}$ we set $S(\mathfrak{a})=S\cap M_m(\mathfrak{a})$. Denote $e_{\infty}(z):=\mathrm{exp}(2\pi iz)$ for $z\in\C$ and $\lambda=\frac{1}{2}\mathrm{tr}$. For $f\in M_k(\Gamma)$ and $\gamma\in G$, there is a Fourier-Jacobi expansion of the form
\[
(f|_k\gamma)(z)=\sum_{\tau\in S_+}c(\tau,\gamma,f;v,w)e(\lambda(\mathfrak{i}(\tau)u)),z=(u,v,w)\in\mathfrak{Z},
\]
In particular, for $\gamma=1$ we simply write
\[
f(z)=\sum_{\tau\in S_+}c(\tau;v,w)e(\lambda(\mathfrak{i}(\tau)u)),
\]
We call $f$ a cusp form if $c(\tau,\gamma,f;v,w)=0$ for every $\gamma\in G$ and every $\tau$ such that $\det(h)=0$. The space of cusp form is denoted by $S_k(\Gamma)$.

Given a function $\mathbf{f}:G(\mathbb{A})\to\C$, we can, by abusing the notation, also view it as a function $\mathbf{f}:G(\mathbb{A}_{\mathbf{h}})\times\mathfrak{Z}\to\C$ by setting $\mathbf{f}(g_{\mathbf{h}},z):=j(g_z,z_0)^k\mathbf{f}(g_{\mathbf{h}}g_z)$ with $z=g_z\cdot z_0$.

\begin{defn}
A function $\mathbf{f}:G(\mathbb{A})\to\C$ is called a quaternionic modular form of weight $k$, level $\mathfrak{n}$ if\\
(1) Viewed as a function $\mathbf{f}:G(\mathbb{A}_{\mathbf{h}})\times\mathfrak{Z}\to\C$, $\mathbf{f}(g_{\mathbf{h}},z)$ is holomorphic in $z$,\\
(2) For $\alpha\in G(\Q),k_{\infty}\in K_{\infty},k\in K_1(\mathfrak{n})$
\[
\mathbf{f}(\alpha gk_{\infty}k)=j(k_{\infty},z_0)^{-k}\mathbf{f}(g), 
\]
or equivalently, \newline
(2') Viewed as a function $\mathbf{f}:G(\mathbb{A}_{\mathbf{h}})\times\mathfrak{Z}\to\C$, for $\alpha\in G(\Q),k\in K_1(\mathfrak{n})$ we have
\[
\mathbf{f}(\alpha g_{\mathbf{h}}k,\alpha z)=j(\alpha,z)^k\mathbf{f}(g_{\mathbf{h}},z).
\]
We will denote the space of such functions by $\mathcal{M}_k(K_1(\mathfrak{n}))$.
\end{defn}

We call $\mathbf{f}\in\mathcal{M}_k(K_1(\mathfrak{n}))$ a cusp form if 
\[
\int_{U(\Q)\backslash U(\mathbb{A})}\mathbf{f}(ug)\mathbf{d}u=0,
\]
for all unipotent radicals $U$ of all proper parabolic subgroups of $G$. The space of cusp forms will be denoted by $\mathcal{S}_k(K_1(\mathfrak{n}))$. 

It is well known that above two definitions are related by
\[
\mathcal{M}_k(K_1(\mathfrak{n}))\cong\bigoplus_{j}M_k(\Gamma^j_1(N)),\mathcal{S}_k(K_1(\mathfrak{n}))\cong\bigoplus_{j}S_k(\Gamma^j_1(N)),
\]

Write $\mathbf{f}\leftrightarrow(f_0,f_1,...,f_h)$ for the correspondence under above maps. Here, $f_j(z)=\mathbf{f}(t_j,z)=j(g_z,i)^k\mathbf{f}(t_jz)$ with $z=g_z\cdot i$. 

When $n=2m,r=0$ then the Fourier-Jacobi expansion becomes the usual Fourier expansion. For $x\in\Q_v,v\in\mathbf{h}$ define $e_v(x)=e_{\infty}(-y)$ with $y\in\bigcup_{n=1}^{\infty}p^{-n}\Z$ such that $x-y\in\Z_v$. Set $e_{\mathbb{A}}(x)=e_{\infty}(x_{\infty})\prod_{v\in\mathbf{h}}e_v(x_v)$. Let $\mathbf{f}\in\mathcal{M}_k(K_1(\mathfrak{n}))$. For $g\in G(\mathbb{A})$ we write it as $g=\gamma t_jkp_{\infty}k_{\infty}$ with $\gamma\in G(\Q)$, $k\in K_1(\mathfrak{n}),k_{\infty}\in K_{\infty}$. Take $t_j$ of the form $\left[\begin{array}{cc}
q_j & \sigma_j\hat{q}_j\\
0 & \hat{q}_j
\end{array}\right]$ with $q_j\in\GL_m(\mathbb{B}_{\mathbf{h}}),\sigma_j\in S(\mathbb{A}_{\mathbf{h}})$ and $p_{\infty}=\left[\begin{array}{cc}
q_{\infty} & \sigma_{\infty}\hat{q}_{\infty}\\
0 & \hat{q}_{\infty}
\end{array}\right]$ with $q_{\infty}\in\GL_m(\mathbb{B}_{\infty}),\sigma_{\infty}\in S(\R)$. Set $q=q_jq_{\infty},\sigma=\sigma_j\sigma_{\infty}$, then $\mathbf{f}$ has a Fourier expansion of the form
\[
\mathbf{f}(g)=j(k_{\infty},z_0)^{-k}\sum_{\tau\in S}\det(q_{\infty})^{-k}c(\tau,q;\mathbf{f})e_{\infty}(\lambda(q^{\ast}\tau q)z_0)e_{\mathbb{A}}(\lambda(\tau\sigma)).
\]
We call $c(\tau,q;\mathbf{f})$ the Fourier coefficients of $\mathbf{f}$. 

For two modular forms $f,h\in M_k(\Gamma)$, we define the Petersson inner product by
\[
\langle f,h\rangle=\int_{\Gamma\backslash\mathfrak{Z}}f(z)\overline{h(z)}\delta(z)^k\mathbf{d}z,
\]
whenever the integral converges. For example, this is well defined when one of $f,g$ is a cusp form. Adelically, for $\mathbf{f,h}\in\mathcal{M}_k(K_1(\mathfrak{n}))$ we define
\[
\langle\mathbf{f,h}\rangle=\int_{G(\Q)\backslash G(\mathbb{A})/K_1(\mathfrak{n})K_{\infty}}\mathbf{f}(g)\overline{\mathbf{h}(g)}\mathbf{d}g.
\]
Here $\mathbf{d}g$ an invariant differential of $G(\mathbb{A})$ given as follow: $\mathbf{d}g=\mathbf{d}g_{\mathbf{h}}\mathbf{d}g_{\infty}$ where $\mathbf{d}g_{h}$ is the canonical measure on $G(\mathbb{A}_{\mathbf{h}})$ normalized such that the volume of $K_1(\mathfrak{n})$ is $1$ and $\mathbf{d}g_{\infty}=\mathbf{d}(g_{\infty}z_0)$ with $\mathbf{d}z$ an invariant differential of $\mathfrak{Z}$.

Viewing $\mathbf{f,h}$ as functions $\mathbf{f,g}:G(\mathbb{A}_{\mathbf{h}})\times\mathfrak{Z}\to\C$ we have,
\[
\langle\mathbf{f,h}\rangle=\int_{G(\Q)\backslash (G(\mathbb{A}_{\mathbf{h}})/K_1(\mathfrak{n})\times\mathfrak{Z})}\mathbf{f}(g,z)\overline{\mathbf{h}(g,z)}\delta(z)^k\mathbf{d}g_{\mathbf{h}}\mathbf{d}z.
\]
Again these integrals are well-defined if one of $\mathbf{f,h}$ is a cusp form. If $\mathbf{f}\leftrightarrow(f_j)_j,\mathbf{h}\leftrightarrow(h_j)_j$ then
\[
\langle\mathbf{f,h}\rangle=\sum_{j}\langle f_j,h_j\rangle.
\]

\subsection{Hecke operators and L-functions}

In the rest of the paper, we make the assumption that all finite places $v$ with $v\nmid\mathfrak{n}$ are split in $\mathbb{B}$. We define the groups
\[
E=\prod_{v\in\mathbf{h}}\GL_m(\mathcal{O}_v),\,\,\,M=\{x\in\GL_m(\mathbb{B})_{\mathbf{h}}:x_v\in M_m(\mathcal{O}_v)\},
\]
Let $\mathfrak{X}=\prod_v\mathfrak{X}_v$ be a subgroup of $G(\mathbb{A}_{\mathbf{h}})$ with $\mathfrak{X}_v=K_v$ if $v|\mathfrak{n}$ and $\mathfrak{X}_v=G_v$ if otherwise. Define the Hecke algebra $\mathbb{T}=\mathbb{T}(K_1(\mathfrak{n}),\mathfrak{X})$ be the $\Q$-algebra generated by double coset $[K_1(\mathfrak{n})\xi K_1(\mathfrak{n})]$ with $\xi\in\mathfrak{X}$. Given $\mathbf{f}\in\mathcal{M}_k(K_1(\mathfrak{n}))$ the Hecke operator $[K_1(\mathfrak{n})\xi K_1(\mathfrak{n})]$ acts on $\mathbf{f}$ by
\[
(\mathbf{f}|[K_1(\mathfrak{n})\xi K_1(\mathfrak{n})])(g)=\sum_{y\in Y}\mathbf{f}(gy^{-1}),
\]
where $Y$ is a finite subset of $G_{\mathbf{h}}$ such that 
\[
K_1(\mathfrak{n})\xi K_1(\mathfrak{n})=\bigcup_{y\in Y}K_1(\mathfrak{n})y.
\]
We say that $\mathbf{f}\in\mathcal{S}_k(K_1(\mathfrak{n}))$ is an eigenform if there exists some numbers $\lambda_{\mathbf{f}}(\xi)\subset\C$ called eigenvalues such that
\[
\mathbf{f}|[K_1(\mathfrak{n})\xi K_1(\mathfrak{n})]=\lambda_{\mathbf{f}}(\xi)\mathbf{f}\text{ for all }\xi\in\mathfrak{X}.
\]
We use the notation $l(\xi):=\det(r)$ if $\xi\in G(\mathcal{O}_{\mathbf{h}})\mathrm{diag}[\hat{r},1,r]G(\mathcal{O}_{\mathbf{h}})$ with $r\in M$. For a Hecke character $\chi$ we define the series
\[
D(s,\mathbf{f},\chi)=\sum_{\xi\in K_1(\mathfrak{n})\backslash\mathfrak{X}/K_1(\mathfrak{n})}\lambda_{\mathbf{f}}(\xi)\chi^{\ast}(l(\xi))l(\xi)^{-s},\text{ }\mathrm{Re}(s)\gg0.
\]
Here $\chi^{\ast}$ is the associated Dirichlet character of $\chi$. We further define the L-function by
\begin{equation}\label{Relation of Dirichlet series}
L(s,\mathbf{f},\chi)=\Lambda_{\mathfrak{n}}(s,\chi)D(s,\mathbf{f},\chi),\Lambda_{\mathfrak{n}}(s,\chi)=\prod_{i=0}^{n-1}L_{\mathfrak{n}}(2s-2i,\chi^2).
\end{equation}
Here the subscript $\mathfrak{n}$ means the Euler factors at $v|\mathfrak{n}$ are removed.

Define $\mathbb{T}_v=\mathbb{T}(K_v,\mathfrak{X}_v)$ be the local counterpart of $\mathbb{T}$ so $\mathbb{T}=\otimes'_v\mathbb{T}_v$. Obviously, $\mathbb{T}_v$ is trivial if $v|\mathfrak{n}$. For $v\nmid\mathfrak{n}$, by our assumptions it splits so we can identify the local group $G_v$ with local orthogonal group. Such local Hecke algebra is discussed in \cite{Sh04} where following Satake map was constructed
\[
\omega:\mathbb{T}_v\to\Q[t_1,...,t_{n},t_1^{-1},...,t_n^{-1}].
\]
Given an eigenform $\mathbf{f}$, the map $\xi\mapsto\lambda_{\mathbf{f}}(\xi)$ induces homomorphism $\mathbb{T}_v\to\C$ which are parametrised by Satake parameters
\[
\alpha_{1,v}^{\pm 1},...,\alpha_{n,v}^{\pm 1}.
\]
The L-function then has an Euler product expression
\[
L(s,\mathbf{f},\chi)=\prod_{p\nmid\mathfrak{n}}L_p(s,\mathbf{f},\chi),
\]
with $L_p(s,\mathbf{f},\chi)$ given by
\[
(1-p^{2n-2}\chi^2(p))^{-1}\prod_{i=1}^n\left((1-\alpha_{i,p}\chi(p)p^{n-1-s})(1-\alpha_{i,p}^{-1}\chi(p)p^{n-1-s})\right)^{-1},
\]
where $p$ is a prime corresponding to some place $v$ in the notation above.

\section{Eisenstein Series and Integral Representation of L-functions}

\subsection{Siegel Eisenstein Series and its Fourier Expansion}

We fix an integer $0\leq t\leq m$, and for $x\in G$ we write  
\[
x=\left[\begin{array}{ccccc}
a_1 & a_2 & b_1 & c_1 & c_2\\
a_3 & a_4 & b_2 & c_3 & c_4\\
g_1 & g_2 & e & f_1 & f_2\\
h_1 & h_2 & l_1 & d_1 & d_2\\
h_3 & h_4 & l_2 & d_3 & d_4
\end{array}\right]
\]
with block size $(t,m-t,r,t,m-t)\times(t,m-t,r,t,m-t)$. The $t$-Klingen parabolic subgroup of $G_{n}$ is defined as
\[
P_{n}^t=\{x\in G:a_2=g_2=h_2=h_3=h_4=l_2=d_3=0\}.
\]
Clearly we have $P_n^m=G_{n}$. We define a projection map $\pi_t:P_n^t\to G_{2t+r}$ by
\[
\left[\begin{array}{ccccc}
a_1 & 0 & b_1 & c_1 & c_2\\
a_3 & a_4 & b_2 & c_3 & c_4\\
g_1 & 0 & e & f_1 & f_2\\
h_1 & 0 & l_1 & d_1 & d_2\\
0 & 0 & 0 & 0 & d_4
\end{array}\right]\mapsto\left[\begin{array}{ccc}
a_1 & b_1 & c_1\\
g_1 & e & f_1\\
h_1 & l_1 & d_1
\end{array}\right].
\]

In particular, if $r=0,n=2m$. The parabolic subgroup 
\[
P_n:=P^0_n=\left\{\left[\begin{array}{cc}
a & b\\
c & d
\end{array}\right]\in G_{n}:c=0\right\},
\]
is called Siegel parabolic subgroup. We now fix weight $l \in \mathbb{N}$ and let $\chi$ be a Hecke character whose conductor divides $\mathfrak{n}$. The Siegel-type Eisenstein series is defined as
\[
\mathbf{E}_l(x,s)=\mathbf{E}^m_l(x,s;\chi)=\sum_{\gamma\in P_m\backslash G_{2m}}\varphi(\gamma x,s),
\]
with 
\[
\varphi(x,s)=\chi_{\mathbf{h}}(\det(d_p))^{-1}j(x,z_0)^{-l}|\det(d_p)|^{-s}_{\mathbf{h}}|j(x,z_0)|^{l-s},
\]
if $x=pkk_{\infty}\in P_n(\mathbb{A})K_1(\mathfrak{n})K_{\infty}$ and $\varphi(x,s)=0$ if otherwise. Let $J_{\mathbf{h}}\in G_{n}(\mathbb{A})$ be an element defined by $J_v=J_m$ for $v\in\mathbf{h}$ and $J_{\infty}=1$. We set $\mathbf{E}_l^{\ast}(x,s)=\mathbf{E}_l(xJ_{\mathbf{h}}^{-1},s)$. Then we have the following Proposition (see for example \cite{B}).

\begin{prop}
$\mathbf{E}_l^{\ast}(x,s)$ has a Fourier expansion of form
\[
\mathbf{E}_l^{\ast}\left(\left[\begin{array}{cc}
q & \sigma\hat{q}\\
0 & \hat{q}
\end{array}\right],s\right)=\sum_{h\in S}c(h,q,s)e_{\mathbb{A}}(\lambda(h\sigma)),
\]
where $q\in\GL_m(\mathbb{B}_{\mathbb{A}})$ and $\sigma\in S(\mathbb{A})$. The Fourier coefficient $c(h,q,s)\neq 0$ only if $(q^{\ast}hq)_v\in T(\Q_v)\cap M_m(\mathfrak{n}_v^{-1})$ for all $v\in\mathbf{h}$. In this case we have
\[
c(h,q,s)=A(n)\chi(\det(q_{\mathbf{h}}))^{-1}\det(q_{\infty})^s|\det(q)|_{\mathbf{h}}^{2m-1-s}\alpha_{\mathfrak{n}}(q_{\mathbf{h}}^{\ast}hq_{\mathbf{h}},s,\chi)\xi(q_{\infty}q_{\infty}^{\ast},h,s+l,s-l).
\]
Here\\
(1) $A(n)\in\overline{\Q}^{\times}$ is a constant depending on $n$.\\
(2) If $h$ has rank $r$ then
\[
\alpha_{\mathfrak{n}}(q^{\ast}hq,s,\chi)=\frac{\prod_{i=1}^{m-r}L_{\mathfrak{n}}(2s-4m+2r+2i+1,\chi^2)}{\prod_{i=0}^{m-1}L_{\mathfrak{n}}(2s-2i,\chi^2)}\prod_{p}P_{h,q,p}(\chi^{\ast}(p)p^{-s}),
\]
where $P_{h,q,p}(X)\in\mathcal{O}[X]$ and $P_{h,q,p}=1$ if $\det(h)\in 2^{m+1}\Z_p^{\times}$. Here $p$ is the prime corresponding to $v$.\\
(3) Let $p$ be the number of positive eigenvalues of $h$, $q$ the number of negative eigenvalues of $h$, $t=m-p-q$, then for $y\in S_{\infty}^+,h\in S_{\infty}$
\[
\xi(y,h,s+l,s-l)=\frac{\Gamma_t(2s-2m+1)}{\Gamma_{m-q}(s+l)\Gamma_{m-p}(s-l)}\omega(y,h,s+l,s-l),
\]
where
\[
\Gamma_m(s)=\pi^{m(m-1)}\prod_{i=0}^{m-1}\Gamma(s-2i),m\in\Z,
\]
and $\omega$ is holomorphic with respect to $s+l,s-l$. In particular, when $p=m$,
\[
\xi(y,h,2l,0)=2^{2-2m}(2\pi i)^{2ml}\Gamma_m^{-1}(2l)\det(h)^{l-\frac{2m-1}{2}}e(i\lambda(hy)).
\]
\end{prop}

For $g\in G(\mathbb{A})$, write it as $g=\gamma t_jkp_{\infty}k_{\infty}$ with $\gamma\in G(\Q),k\in \eta K_0(\mathfrak{n})\eta^{-1},k_{\infty}\in K_{\infty}$ and $t_j=\left[\begin{array}{cc}
q_j & \sigma_j\hat{q}_j\\
0 & \hat{q}_j
\end{array}\right],p_{\infty}=\left[\begin{array}{cc}
q_{\infty} & \sigma_{\infty}\hat{q}_{\infty}\\
0 & \hat{q}_{\infty}
\end{array}\right]$. Set $q=q_jq_{\infty},\sigma=\sigma_j\sigma_{\infty}$ then by modularity property, $\mathbf{E}_l^{\ast}(g,s)$ can be written as
\[
\mathbf{E}_l^{\ast}(g,s)=j(k_{\infty},z_0)^{-l}\sum_{h}c(h,q,s)e_{\mathbb{A}}(\lambda(h\sigma)),
\]
with $c(h,q,s)$ in above propositions. We are interested in the special values $\mathbf{E}_l^{\ast}(g,s)$ for $s=l$. From the Fourier expansion, we have the following proposition by counting poles and degree of $\pi$ in those Gamma functions.

\begin{prop}
Assume $l>n-1$. Then the Fourier coefficients $c(h,q,l)\neq 0$ unless $h>0$ and in this case
\[
\det(q_{\infty})^{-l}c(h,q,l)=C\cdot e_{\infty}(i\lambda(q^{\ast}hq)).
\]
Here, up to a constant in $\overline{\Q}$, $C=\chi(\det(q_{\mathbf{h}}))^{-1}|\det(q)|_{\mathbf{h}}^{2m-1-l}$. In particular, $\mathbf{E}_l^{\ast}(g,l)$ is holomorphic in the sense that when viewed as a function on $G(\mathbb{A}_{\mathbf{h}})\times\mathfrak{Z}$, $\mathbf{E}_l^{\ast}(g_{\mathbf{h}},z,l)$ is holomorphic in $z$. 
\end{prop}

We note in particular that the proposition implies that also $\mathbf{E}_l(g,l)$ is holomorphic in the above sense since $\mathbf{E}_l^{\ast}(x,s)=\mathbf{E}_l(xJ_{\mathbf{h}}^{-1},s)$.

\subsection{Coset Decompositions}

Let 
\[
\rho:G_n\times G_n\to G_{N},g_1\times g_2\mapsto R^{-1}\mathrm{diag}[g_1,g_2]R,
\]
be the doubling embedding defined before. To ease the notation we may omit the subscript $n$. Denote $P_N$ for the Siegel parabolic subgroup of $G_N$ and $P^t=P_{n}^t$ the $t$-parabolic subgroup of $G$.

\begin{prop} \label{Coset Decomposition}
For $0\leq t\leq m$, let $\tau_t$ be the element of $G_N$ given by
\[
\tau_i=\left[\begin{array}{cccccc}
1_{m} & 0 & 0 & 0 & 0 & 0\\
0 & 1_r & 0 & 0 & 0_r & 0\\
0 & 0 & 1_{m} & 0 & 0 & 0\\
0 & 0 & e_t & 1_{m} & 0 & 0\\
0 & 0_r & 0 & 0 & 1_r & 0\\
e_t^{\ast} & 0 & 0 & 0 & 0 & 1_{m}
\end{array}\right],e_t=\left[\begin{array}{cc}
1_t & 0\\
0 & 0
\end{array}\right]\in\mathbb{B}^{m}_{m}.
\]
Then $\tau_i$ form a complete set of representatives of $P_N\backslash G_N/G\times G$.
\end{prop}

\begin{proof}
This can be proved similarly to the proof of lemmata 4.1,4.2 in \cite{Sh95}. 
Let $W=\{w\in\mathbb{B}^N_{2N}:wJ_nw^{\ast}=0,\mathrm{rank}(w)=N\}$ then $P_N\backslash G_N\cong \GL_N\backslash W$. Therefore, it suffices to find representatives of $\GL_N\backslash W/G\times G$. Let $w=\left[\begin{array}{cccccc}
a_1 & b_1 & c_1 & a_2 & b_2 & c_2
\end{array}\right]\in W$ of column size $(m,r,m,m,r,m)$. The  condition $wJ_nw^{\ast}=0$ is equivalent to $wR^{\ast}\omega Rw^{\ast}=0$. Explicitly, $wR^{\ast}=\left[\begin{array}{cc}x& y\end{array}\right]$ with 
\[
x=\left[\begin{array}{ccc}
a_1 & \frac{b_1}{2}+b_2\zeta^{-1} & a_2
\end{array}\right],y=\left[\begin{array}{ccc}
-c_1 & -\frac{b_1}{2}+b_2\zeta^{-1} & c_2
\end{array}\right],x\phi_1 x^{\ast}=y\phi_2 y^{\ast}.
\]
Multiplying by some element in $\GL_N$ we can assume 
\[
x\phi_1 x^{\ast}=y\phi_2 y^{\ast}=\left[\begin{array}{ccc}
0 & 0 & e_t\\
0 & \zeta I_r & 0\\
e_t^{\ast} &0 & 0
\end{array}\right]\text{ or }\left[\begin{array}{ccc}
0 & 0 & e_t\\
0 & 0_r & 0\\
e_t^{\ast} &0 & 0
\end{array}\right],0\leq t\leq m.
\]
Let $V=\mathbb{B}^{n}$ with standard basis $\{\epsilon_i\}$ and denote $x_i$ be the $i$-th row of $x$.  Let $\tilde{U}$ be the subspace of $V$ spanned by basis $\epsilon_i$ with $i\leq t$ or $m+r\leq i\leq m+r+t$ and $\tilde{U}^{\perp}$ the subspace spanned by other basis so $V=\tilde{U}\oplus \tilde{U}^{\perp}$. Let $\theta$ be the restriction of $\phi_1$ on $\tilde{U}$ and $\eta$ the restriction on $\tilde{U}^{\perp}$ then we can write $(V,\phi_1)=(\tilde{U},\theta)\oplus(\tilde{U}^{\perp},\eta)$. Assume $x\phi x^{\ast}$ is given by first matrix as above, let $U$ be the subspace of $V$ spanned by vector $x_i$ with $i\leq m$ or $m+r\leq i\leq m+r+t$ and $U'$ be the subspace spanned by other $x_i$. We also denote $U^{\perp}=\{v\in V:u\phi v^{\ast}=0\text{ for any }u\in U\}$ then $V=U\oplus U^{\perp}$. Then there exists an automorphism $\gamma$ of $(V,\phi)$ such that $U\gamma=\tilde{U},U^{\perp}\gamma=\tilde{U}^{\perp}$. Now $U'\gamma\subset\tilde{U}^{\perp}$ is a totally $\eta$-isotropic subspace thus there exists an automorphism $\gamma'$ of $(\tilde{U}^{\perp},\eta)$ such that $U'\gamma\gamma'\subset\sum_{m+t+r+1\leq i\leq n}\mathbb{B}\epsilon_i$. Viewing $\gamma'$ as an automorphism of $(V,\phi)$ and put $g_1=\gamma\gamma'$ we have (similarly for $y$)
\[
xg_1=\left[\begin{array}{ccccc}
1_t & 0 & 0 & 0 & 0\\
0 & 0 & 0 & 0 & u\\
0 & 0 & 1 & 0 & 0\\
0 & 0 & 0 & 1_t & 0\\
0 & 0 & 0 & 0 & v
\end{array}\right],yg_2=\left[\begin{array}{ccccc}
1_t & 0 & 0 & 0 & 0\\
0 & 0 & 0 & 0 & u'\\
0 & 0 & 1 & 0 & 0\\
0 & 0 & 0 & 1_t & 0\\
0 & 0 & 0 & 0 & v'
\end{array}\right].
\]
We further modify
\[
g_1\mapsto g_1\left[\begin{array}{ccccc}
0 & 0 & 0 & 1_t & 0\\
0 & 1 & 0 & 0 & 0\\
0 & 0 & 1 & 0 & 0\\
-1_t & 0 & 0 & 0 & 0\\
0 & 0 & 0 & 0 & 1
\end{array}\right]\text{ so }xg_1=\left[\begin{array}{ccccc}
0 & 0 & 0 & -1_t & 0\\
0 & 0 & 0 & 0 & u\\
0 & 0 & 1 & 0 & 0\\
1_t & 0 & 0 & 0 & 0\\
0 & 0 & 0 & 0 & v
\end{array}\right].
\]
Therefore
\[
w(g_1\times g_2)=\left[\begin{array}{cccccccccc}
0 & 0 & 0 & -1_t & 0 & -1 & 0 & 0 & 0 & 0\\
0 & 0 & 0 & 0 & 0 & 0 & u & 0 & 0 & u'\\
0 & 0 & 0 & 0 & 0 & 0 & 0 & -\zeta & 0 & 0\\
1_t & 0& 0 & 0 & 0 & 0 & 0 & 0 & 1 & 0\\
0 & 0 & 0 & 0 & 0 & 0 & v & 0 & 0 & v'
\end{array}\right].
\]
By our assumption $\mathrm{rank}(w)=N$, we must have $\left[\begin{array}{cc}u & u'\\
v & v'\end{array}\right]$ is of full rank. If $x\phi_1x^{\ast}$ equals the second matrix as above, then by the same argument we can obtain a similar result but without the term $1$ in the middle of $xg_1$ and $yg_2$ which contradicts the assumption $\mathrm{rank}(w)=N$. Suppose $\left[\begin{array}{cc}
a & b\\
c & d
\end{array}\right]=\left[\begin{array}{cc}u & u'\\
v & v'\end{array}\right]^{-1}$, multiplying
\[
\left[\begin{array}{ccccc}
-1_t & 0 & 0 & 0 & 0\\
0 & a & 0 & 0 & b\\
0 & 0 & -\zeta^{-1} & 0 & 0\\
0 & 0 & 0 & 1_t & 0\\
0 & c & 0 & 0 & d
\end{array}\right]\in\GL_N,
\]
on the left we then
get the desired form in the proposition.
\end{proof}

Put $V_t=\tau_t(G_{n}\times G_{n})\tau_t^{-1}\cap P_N$. Then by straightforward computation
\[
V_t=\{\beta\times\gamma\in P_{n}^t\times P_{n}^t:\kappa_t\pi_t(\beta)=\pi_t(\gamma)\kappa_t\},\kappa_t=\left[\begin{array}{ccc}
0 & 0 & -1_t\\
0 & 1_r & 0\\
1_t & 0 & 0
\end{array}\right].
\]
To simplify the computation, we may also use the modified representatives $\tilde{\tau}_t=\tau_t(1_{n}\times(\kappa_t\times 1_{2m-2t}))$ for $0\leq t\leq m$. Put $\tilde{V}_t=\tilde{\tau}_t(G_{n}\times G_{n})\tilde{\tau}_t^{-1}\cap P_N$. Then
\[
\tilde{V}_t=\{\beta\times\gamma\in P_{n}^t\times P_{n}^t:\pi_t(\beta)=\pi_t(\gamma)\}.
\]
One easily shows that (compare with Lemma 4.3 in \cite{Sh95}):
\[
P_{n}^t\times P_{n}^t=\bigcup_{\xi\in G_{2t+r}}\tilde{V}_t((\xi\times 1_{2m-2t})\times 1_{2n})=\bigcup_{\xi\in G_{2t+r}}\tilde{V}_t(1_{2n}\times(\xi\times 1_{2m-2t})),
\]
\[
P^N\tilde{\tau}_t(G_{n}\times G_{n})=\bigcup_{\xi,\beta,\gamma}P^N\tilde{\tau}_t((\xi\times 1_{2m-2t})\beta\times\gamma)=\bigcup_{\xi,\beta,\gamma}P^N\tilde{\tau}_t(\beta\times(\xi\times 1_{2m-2t})\gamma).
\]
where $\xi$ runs over $G_{2t+r}$, $\beta,\gamma$ runs over $P_{n}^t\backslash G_{n}$. In particular,
\[
P^N\tilde{\tau}_{m}(G_{n}\times G_{n})=\bigcup_{\xi\in G_{n}}P^N\tilde{\tau}_{m}(\xi\times 1_{n})=\bigcup_{\xi\in G_{n}}P^N\tilde{\tau}_{m}(1_n\times\xi).
\]

We denote $K_1(\mathfrak{n})=\prod_vK_v$ be the open compact subgroup of $G_n(\mathbb{A}_{\mathbf{h}})$ defined before. We also consider the open compact subgroup of $G_N$,
$K_{1}^N(\mathfrak{n})=\prod_{v \in \mathbf{h}}K_{v}^N$ with,
\[
K_{v}^N(\mathfrak{n})=\left\{\gamma=\left[\begin{array}{cc}
a & b\\
c & d
\end{array}\right]\in G_N(\mathbb{A}_{v})\cap M_N(\mathcal{O}_v):\gamma\equiv 1_N\text{ mod }\mathfrak{n}_v\right\}.
\]

\begin{prop}
Assume $\mathfrak{n}$ is coprime to $(2)$ and $(\zeta)$. Then for $v|\mathfrak{n}$, $\tilde{\tau}_m(\xi\times 1)\tilde{\tau}_m^{-1}\in P_N(\Q_v)K_v^N$ if and only if $\xi\in K_v$.
\end{prop}

\begin{proof}
Let $\xi=\left[\begin{array}{ccc}
a & b & c\\
g & e & f\\
h & l & d
\end{array}\right]\in G_n$ where blocks has size $(m,r,m)\times(m,r,m)$. We calculate that
\[
\tilde{\tau}_m(\xi\times 1)\tilde{\tau}_m^{-1}=\left[\begin{array}{cccccc}
 & & & & & \\
 & & & & & \\
 & & & & & \\
-h & -l & -d+1 & d & \frac{l\zeta}{2} & 0\\
-\zeta^{-1}g & -\zeta^{-1}(e-1) & -\zeta^{-1}f & \zeta^{-1}f & \frac{e+1}{2} & 0\\
a-1 & b & c & -c & \frac{-b\zeta}{2} & 1
\end{array}\right].
\]
Suppose $\tilde{\tau}_m(\xi\times 1)\tilde{\tau}_m^{-1}\in P_N(\Q_v)K^N_v$, then there exist $p\in P_N(\Q_v)$, say
\[
\left[\begin{array}{cccccc}
 & & & & & \\
 & & & & & \\
 & & & & & \\
0 & 0 & 0 & p_{11} & p_{12} & p_{13}\\
0 & 0 & 0 & p_{21} & p_{22} & p_{23}\\
0 & 0 & 0 & p_{31} & p_{31} & p_{33}
\end{array}\right],
\]
such that $p\tilde{\tau}_m(\xi\times 1)\tilde{\tau}_m^{-1}\in K^N_v$. It is obvious that $p_{13}\equiv p_{23}\equiv 0\text{ mod }\mathfrak{c}$ and $p_{33}\equiv 1\text{ mod }\mathfrak{c}$. Comparing the third and fourth columns we have
\[
p_{11}(-d+1)+p_{12}(-\zeta^{-1}f)+p_{13}c\equiv 0\text{ mod }\mathfrak{n},p_{11}d+p_{12}(\zeta^{-1}f)+p_{13}(-c)\equiv 1\text{ mod }\mathfrak{n};
\]
\[
p_{21}(-d+1)+p_{22}(-\zeta^{-1}f)+p_{23}c\equiv 0\text{ mod }\mathfrak{n},p_{21}d+p_{22}(\zeta^{-1}f)+p_{23}(-c)\equiv 0\text{ mod }\mathfrak{n};
\]
\[
p_{31}(-d+1)+p_{32}(-\zeta^{-1}f)+p_{33}c\equiv 0\text{ mod }\mathfrak{n},p_{31}d+p_{32}(\zeta^{-1}f)+p_{33}(-c)\equiv 0\text{ mod }\mathfrak{n}.
\]
This forces $p_{11}\equiv 1\text{ mod }\mathfrak{n}$ and $p_{21}\equiv p_{31}\equiv 0\text{ mod }\mathfrak{n}$. Comparing the second and fifth column and using our assumption on $\mathfrak{n}$ we have
\[
p_{11}(-l)+p_{12}(-\zeta^{-1}(e-1))+p_{13}b\equiv0\text{ mod }\mathfrak{n},p_{11}(l\zeta)+p_{12}(e+1)+p_{13}(-b\zeta)\equiv0\text{ mod }\mathfrak{n};
\]
\[
p_{21}(-l)+p_{22}(-\zeta^{-1}(e-1))+p_{23}b\equiv0\text{ mod }\mathfrak{n},p_{21}(l\zeta)+p_{22}(e+1)+p_{23}(-b\zeta)\equiv2\text{ mod }\mathfrak{n};
\]
\[
p_{31}(-l)+p_{32}(-\zeta^{-1}(e-1))+p_{33}b\equiv0\text{ mod }\mathfrak{n},p_{31}(l\zeta)+p_{32}(e+1)+p_{33}(-b\zeta)\equiv0\text{ mod }\mathfrak{n}.
\]
The second line shows that $p_{22}\equiv e\equiv 1\text{ mod }\mathfrak{n}$ and then $p_{12}\equiv p_{32}\equiv 0\text{ mod }\mathfrak{n}$ by other formulas. Therefore, $p\in P_N(\Q_v)\cap K^N_v$. From above identities we already have 
\[
d-1\equiv f\equiv c\equiv l\equiv e-1\equiv b\equiv 0\text{ mod }\mathfrak{n}.
\]
The claim $\xi\in K_v$ then follows from above identities together with
\[
p_{i1}(-h)+p_{i2}(-\zeta^{-1}g)+p_{i3}(a-1)\equiv 0\text{ mod }\mathfrak{n},i=1,2,3.
\]
\end{proof}

\subsection{Integral Representation}

Let $\chi$ be a Hecke character whose conductor divides the fixed ideal $\mathfrak{n}$. Recall that we write $N=2n$. We define $\sigma \in G_N(\mathbb{A}_{\mathbf{h}})$ as $\sigma_v := 1$, the identity matrix, if  $v\nmid\mathfrak{n}$ and
$\sigma_v = \tilde{\tau}_m$ if $v|\mathfrak{n}$.

We then consider the weight $k$ Siegel-type Eisenstein series (twisted by $\tilde{\tau}_m$) on $G_N$ defined as
\[
\mathbf{E}(\mathfrak{g},s):=E^n_k(\mathfrak{g}\,\sigma^{-1},s;\chi),\,\,\,\,\mathfrak{g} \in G_N(\mathbb{A}),
\]

\begin{comment}
\[
\mathbf{E}(\mathfrak{g},s):=E_k^{\tilde{\tau}_m}(\mathfrak{g},s;\chi,\mathfrak{n}):=\sum_{\gamma\in P_N\backslash G_N}\phi(\gamma\mathfrak{g},s),\mathfrak{g}\in G_N
\]
where for $v\nmid\mathfrak{n}$, $\phi$ is supported on $P_N(\mathbb{A})K^N_1(\mathfrak{n})K^N_{\infty}$ in $\mathfrak{g}$ and for $\mathfrak{g}=pk_1k_{\infty}$ with $p\in P_N(\mathbb{A}),k_1\in K^N_1(\mathfrak{n}),k_{\infty}\in K^N_{\infty}$,
\[
\phi(\mathfrak{g},s)=\chi_v(\det(d_p))^{-1}|\det(d_p)|_{v}^{-s}j(\mathfrak{g},z_0)^{-k}|j(\mathfrak{g},z_0|^{k-s}.
\]
and for $v|\mathfrak{n}$, $\phi$ is supported on $P_N(\mathbb{A})K^N_1(\mathfrak{n})K^N_{\infty}\tilde{\tau}_m$ in $\mathfrak{g}$ and for $\mathfrak{g}=pk_1k_{\infty}\tilde{\tau}_m$ with $p\in P_N(\mathbb{A}),k_1\in K^N_1(\mathfrak{n}),k_{\infty}\in K^N_{\infty}$,
\[
\phi(\mathfrak{g},s)=\chi_v(\det(d_p))^{-1}|\det(d_p)|_{v}^{-s}j(\mathfrak{g},z_0)^{-k}|j(\mathfrak{g},z_0)|^{k-s}.
\]
\end{comment}
By decomposition of $G_N$ we can write
\[
\mathbf{E}(\mathfrak{g},s)=\sum_{t=0}^m\sum_{\gamma\in P_N\backslash P_N\tilde{\tau}_t(G\times G)}\phi(\gamma\mathfrak{g},s)=:\sum_{t=0}^m\mathbf{E}_t(\mathfrak{g},s).
\]
Assume $\mathfrak{g}=h\times g$ with $h,g\in G(\mathbb{A})$ and let $\mathbf{f}\in S_k(K_1(\mathfrak{n}))$ be a cusp form. 

We are going to study the integral
\[
\int_{G(\Q)\backslash G(\mathbb{A})/K_1(\mathfrak{n})K_{\infty}}\mathbf{E}(g\times h,s)\mathbf{f}(h)\mathbf{d}h.
\]

\begin{prop}
Assume $t<m$, then 
\[
\int_{G(\Q)\backslash G(\mathbb{A})/K_1(\mathfrak{n})K_{\infty}}\mathbf{E}_t(g\times h,s)\mathbf{f}(h)\mathbf{d}h=0.
\]
\end{prop}

\begin{proof}
Let $U_t$ be the unipotent radical of $P^t_n$. Then the integral equals
\[
\int_{G(\Q)\backslash G(\mathbb{A})/K_1(\mathfrak{n})K_{\infty}}\sum_{\xi\in G_{2t+r}}\sum_{\beta,\gamma\in P^t_n\backslash G}\phi(\tilde{\tau}_t((\xi\times 1_{2n-2t})\beta g\times\gamma h),s)\mathbf{f}(h)\mathbf{d}h
\]
\[
=\int_{P^t_n\backslash G(\mathbb{A})/K_1(\mathfrak{n})K_{\infty}}\sum_{\xi\in G_{2t+r}}\sum_{\gamma\in P^t_n\backslash G_n}\phi(\tilde{\tau}_t((\xi\times 1_{2n-2t}) g\times\gamma h),s)\mathbf{f}(h)\mathbf{d}h
\]
\[
=\int_{U_t(\mathbb{A})P^t_n(\Q)\backslash G(\mathbb{A})/K_1(\mathfrak{n})K_{\infty}}\int_{U_t(\Q)\backslash U_t(\mathbb{A})}\sum_{\xi,\gamma}\phi(\tilde{\tau}_t(\xi ng\times\gamma h),s)\mathbf{f}(nh)\mathbf{d}n\mathbf{d}h.
\]
Since $\xi$ normalizes $U_t$ and $\tilde{\tau}_t(n\times 1)\in P_n$ this equals
\[
=\int_{U_t(\mathbb{A})P^t_n(\Q)\backslash G(\mathbb{A})/K_1(\mathfrak{n})K_{\infty}}\int_{U_t(\Q)\backslash U_t(\mathbb{A})}\sum_{\xi,\gamma}\phi(\tilde{\tau}_t(\xi g\times\gamma h),s)\mathbf{f}(nh)\mathbf{d}n\mathbf{d}h,
\]
which vanishes by the cuspidality of $\mathbf{f}$.
\end{proof}

Therefore
\[
\int_{G(\Q)\backslash G(\mathbb{A})/K_1(\mathfrak{n})K_{\infty}}\mathbf{E}(g\times h,s)\mathbf{f}(h)\mathbf{d}h=\int_{G(\Q)\backslash G(\mathbb{A})/K_1(\mathfrak{n})K_{\infty}}\mathbf{E}_m(g\times h,s)\mathbf{f}(h)\mathbf{d}h.
\]
\[
=\int_{G(\mathbb{A})/K_1(\mathfrak{n})K_{\infty}}\phi(\tilde{\tau}_m(g\times h),s)\mathbf{f}(h)\mathbf{d}h.
\]
The infinite part is calculated in following lemma.

\begin{lem} \label{Reproducing Kernel} For $k + Re(s) > 2n+1$ we have,
\[
\int_{G_{\infty}/K_{\infty}}\phi_{\infty}(\tilde{\tau}_m(g_{\infty}\times h_{\infty}),s)\mathbf{f}(h_{\mathbf{h}}\cdot h_{\infty})\mathbf{d}h_{\infty}=c_k(s)\mathbf{f}(h_{\mathbf{h}}\cdot g_{\infty}),
\]
with

\[
c_k(s) = \alpha(s) \pi^{\frac{n(n-1)}{2}} \frac{\Gamma(s+k-2n+3) \Gamma(s+k-2n+5) \ldots \Gamma(s+k-1)}{\Gamma(s+k-n+2) \Gamma(s+k-n+3) \ldots \Gamma(s+k)},
\]

\end{lem}

where $\alpha(s)$ is a holomorphic function on $s \in \mathbb{C}$ such that $\alpha(\lambda) \in \overline{\mathbb{Q}}$ for all $\lambda \in \mathbb{Q}$.
\begin{proof}
Note that
\[
j(\tilde{\tau}_m(g_{\infty}\times h_{\infty}),z_0\times z_0)=j(\tilde{\tau}_m,g_{\infty}z_0\times h_{\infty}z_0 )j(g_{\infty},z_0)\overline{j(h_{\infty},z_0)}.
\]
Put $w=h_{\infty}z_0,z=g_{\infty}z_0$ then since
\[
f(w)=j(h_{\infty},z_0)^{k}\mathbf{f}(h_{\infty}),j(\tilde{\tau}_m,z\times w)=\delta(w,z),
\]
the integral becomes
\[
j(g_{\infty},z_0)^{-k}\delta(z)^{\frac{s-k}{2}}\int_{\mathfrak{Z}}\delta(w,z)^{-k}|\delta(w,z)|^{k-s}\delta(w)^{\frac{k+s}{2}}f(w)dw.
\]

This kind of integral is calculated in \cite[Appendix A2]{Sh97} and \cite{Hua}. In particular it is shown there that for $k+ Re(s) > n+\frac{1}{2}$
\[
\int_{\mathfrak{Z}}\delta(w,z)^{-k}|\delta(w,z)|^{-2s}\delta(w)^{s+k}f(w) dw = \widetilde{c}_k(s) f(z) \delta(z)^{-s},
\]
where $\widetilde{c}_k(s)$ is a function on $s$ which does not depend on $f$. Indeed as it is explained in \cite{Sh97} the quantity $\widetilde{c}_k(s)$ is independent of $f$ and it is equal to
\[
\widetilde{c}_k(s) = \alpha(s) \int_{\mathfrak{B}} \det(I + z \bar{z})^{s+k} \textbf{d}z,
\] 
where $\textbf{d}z$ is the invariant measure on the bounded domain and is given as $\textbf{d}z = \det(I + z \bar{z})^{-n+1} dz$, and $\alpha(s)$ is a holomorphic function on $s$ such that $\alpha(\lambda) \in \overline{\mathbb{Q}}$ for all $\lambda \in \mathbb{Q}$ (actually it can be made precise but we do not need it here). But this last integral has been computed in \cite[page 46]{Hua} from which we obtain that
\[
\widetilde{c}_k(s) = \alpha(s) \pi^{\frac{n(n-1)}{2}} \frac{\Gamma(2\lambda +1) \Gamma(2 \lambda +3) \ldots \Gamma(2 \lambda +2n-3)}{\Gamma(2\lambda +n) \Gamma(2\lambda +n+1) \ldots \Gamma(2\lambda + 2n -2)},
\]
where $\lambda = s+k -n+1$.

Setting now $s \mapsto \frac{s-k}{2}$ we obtain that.
\[
j(g_{\infty},z_0)^{-k}\delta(z)^{\frac{s-k}{2}}\int_{\mathfrak{Z}}\delta(w,z)^{-k}|\delta(w,z)|^{k-s}\delta(w)^{\frac{k+s}{2}}f(w)dw =
\]
\[
\widetilde{c}_k((s-k)/2)j(g_{\infty},z_0)^{-k}f(z)=c_k(s)\mathbf{f}(h_{\mathbf{h}}\cdot g_{\infty}),
\]
where we have set $c_k(s) := \widetilde{c}_k((s-k)/2)$.
\end{proof}

By changing variables, it remains to calculate
\[
\int_{K_1(\mathfrak{n})\backslash G(\mathbb{A}_{\mathbf{h}})}\phi_{\mathbf{h}}(\tilde{\tau}_m(h\times 1),s)\mathbf{f}(gh^{-1})\mathbf{d}h.
\]
Note that for $v|\mathfrak{n}$, $\phi_{v}$ is nonzero unless $h_{v}\in K_v$. For $v\nmid\mathfrak{n}$, using Cartan decomposition, we can write $h_v=h_1\cdot\mathrm{diag}[\hat{r},1,r]\cdot h_2$ with $h_1,h_2\in G(\mathcal{O}_v)$ and $r\in\GL_m(\mathbb{B}_v)\cap M_m(\mathcal{O}_v)$. By definition of $\phi$,
\[
\phi_v(\tilde{\tau}_m(h_v\times 1))=\chi_{v}(\det(r))|\det(r)|_v^{-s}.
\]
Since $h_v$ runs through $G(\mathcal{O}_v)\backslash G(\mathcal{O}_v)\mathrm{diag}[\hat{r},1,r]G(\mathcal{O}_v)$, comparing with the action of Hecke operators we can write the integral as 
\[
\sum_{\xi\in G(\mathcal{O}_{\mathbf{h}})\backslash\mathfrak{X}/G(\mathcal{O}_{\mathbf{h}})}\mathbf{f}(g)|T_{\xi}\cdot\chi_{\mathbf{h}}(\det(r))|\det(r)|_{\mathbf{h}}^{-s}.
\]
Here $\xi=\mathrm{diag}[\hat{r},1,r]$ and $T_{\xi}$ is the Hecke operator corresponds to double coset $[G(\mathcal{O}_v)\xi G(\mathcal{O}_v)]$. Assume $\mathbf{f}$ is an eigenform such that $\mathbf{f}|T_{\xi}=\lambda_{\mathbf{f}}(\xi)$. Then the sum can be further written as
\[
\sum_{G(\mathcal{O}_{\mathbf{h}})\backslash\mathfrak{X}/G(\mathcal{O}_{\mathbf{h}})}\lambda_{\mathbf{f}}(\xi)\cdot\chi_{\mathbf{h}}(\det(r))|\det(r)|_{\mathbf{h}}^{-s}\mathbf{f}(g)=D(s,\mathbf{f},\chi)\mathbf{f}(g).
\]
We summarize the discussion in following theorem.

\begin{thm} \label{Integral Representation}
Let $\mathbf{f}\in\mathcal{S}_k^n(K_1(\mathfrak{n}))$ be an eigenform and assume $\mathfrak{n}$ is coprime to $(2\zeta)$. Then
\[
\int_{G(\Q)\backslash G(\mathbb{A})/K_1(\mathfrak{n})K_{\infty}}\mathbf{E}(g\times h,s)\mathbf{f}(h)\mathbf{d}h=c_k(s)D(s,\mathbf{f},\chi)\mathbf{f}(g).
\]
%Here $c_k(s)$ is some constant depending on $n,s,k$.
\end{thm}

\section{Algebraic Modular Forms and Differential Operators}

In order to move from the analytic considerations discussed so far to algebraic questions, we need to discuss the notion of an algebraic modular form in our setting.
 The notion of algebraic modular forms on Hermitian symmetric space is well understood. There are mainly four characterizations of algebraic modular forms: via Fourier-Jacobi expansion, CM points, pullback to elliptic modular forms and canonical model of automorphic vector bundle. For example, in \cite[III.7]{M} automorphic forms are interpreted as sections of certain automorphic vector bundles. The canoncial model of  automorphic vector bundles then defines a subspace of algebraic automorphic forms (see also \cite{H1,H2}). It is also proved there that this definition is equivalent to the definition in terms of values at CM points. In \cite{G84}, Garrett gives three characterisations of algebraicity for scalar valued modular forms via CM points, Fourier-Jacobi expansion and pullback to elliptic modular forms. They are also proved to be equivalent.

However, in this work, instead of simply referring to the results of Harris as in \cite{H1,H2} we have decided to offer a definition of algebraic modular forms via CM points using the rather more explicit language of Shimura as in \cite{Sh00}, without need to refer to the more advanced and general theory as developed by Deligne, Milne and others. Indeed our approach of the definition of CM points and the underlying periods follows an idea in the first works of Shimura on the subject (\cite{Sh67}), where one ``tensors'' a given embedding $h: K_1 \times \ldots \times K_n \hookrightarrow G$, of CM field $K_i$, with another CM field $K$, disjoint to the $K_i$'s to obtain a point whose associated abelian variety is of CM type (see also \cite[proof of Theorem 6.4]{Deligne}). In this way we will be able to define and study the CM points in our case by considering an embedding of our group in to a unitary group, after a choice of a imaginary quadratic field. However we will show that our definition of CM points and the attached periods is independent of the choice of the auxiliary imaginary quadratic field.  This should be seen as our main contribution in this section which we believe it is worth appearing in the literature and could be helpful to other researchers, thanks to its rather explicit nature and basic background, Finally we will show that in certain case, when the underlying symmetric space is a tube domain, i.e a Siegel Domain of Type I, our definition is equivalent to standard definition using Fourier expansion.

\subsection{CM points} We introduce the following setting, with some small repetition of what we have discussed so far.

We let $\mathbb{B}$ to be a definite quaternion algebra over $\Q$, $T^{\ast}=-T\in M_n(\mathbb{B})$ a skew-hermitian matrix and define the algebraic group
\[
G:=G(T):=\{g\in\mathrm{SL}_n(\mathbb{B}):g^{\ast}Tg=T\}.
\]

Let $K_i$, $i=1,\ldots,n$ be imaginary quadratic fields and consider the CM algebra $Y=K_1\times...\times K_n$ and $Y^1=\{y\in Y:yy^{\rho}=1\}$ with $\rho$ induced by the nontrivial involutions (i.e. complex conjugation) on each $K_i$. We are interested in embeddings $h:Y^1\to G(T)$. Clearly $h(Y^1)\subset G(T)(\R)$ and $(Y^1 \otimes_{\mathbb{Q}} \mathbb{R})^{\times}$ is a compact subgroup of $G(T)(\R)$. 
 Let us show that there always exists such an embedding.

Without loss of generality we may write $T=\mathrm{diag}[a_1,...,a_n]$ in diagonal form.  We then select as imaginary quadratic fields $K_i:=\Q(a_i)$, for $i=1,\ldots n$, and define the embedding
\[
h:Y^1\to G(T),(y_1,...,y_n)\mapsto\mathrm{diag}[y_1,...,y_n].
\]

Back to our general considerations, we select an imaginary quadratic field $K$ which is different from the $K_i$'s above, and
splits $\mathbb{B}$. It is easy to see that that there exists always such a field $K$. We now fix an embedding $M_n(\mathbb{B})\to M_{2n}(K)$. Denote the image of $T$ in $M_{2n}(K)$ by $\mathcal{T}$ and the unitary group
\[
U(\mathcal{T}):=\{g\in\mathrm{GL}_{2n}(K):g^{\ast}\mathcal{T}g=\mathcal{T}\}.
\] 
We note that the $\R$-group of $U(\mathcal{T})$ is isomorphic to
\[
U(n,n)=\left\{g\in\GL_{2n}(\C):g^{\ast}\left[\begin{array}{cc}
i\cdot 1_n & 0\\
0 & -i\cdot 1_n
\end{array}\right]g=\left[\begin{array}{cc}
i\cdot 1_n & 0\\
0 & -i\cdot 1_n
\end{array}\right]\right\}.
\]
Its action on the bounded domain, (see for example \cite{Sh00}), 
\[
\mathcal{B}=\{z\in M_n(\C):1-z^{\ast}z>0\},
\]
is defined by $gz=(az+b)(cz+d)^{-1}$ for $g=\left[\begin{array}{cc}
a & b\\
c & d
\end{array}\right]$, with the obvious block matrices. The two factors of automorphy are given by $\lambda(g,z)=\overline{c}\transpose{z}+\overline{d}$, and $\mu(g,z)=cz+d.$ The embedding $M_n(\mathbb{B})\to M_{2n}(K)$ induces an embedding $\mathfrak{i}:G(T)\to U(\mathcal{T})$ which is compatible with natural inclusion $\iota:\mathfrak{B}\to\mathcal{B}$. We will view $G(T)$ (resp. $\mathfrak{B}$) as a subgroup (resp. subspace) of $U(\mathcal{T})$ (resp. $\mathcal{B}$) under this embedding.

 \begin{lem}
(1) $Y$ is spanned by $Y^1$ over $\mathbb{Q}$. In particular there exists an element $\beta \in Y^1$ such that $Y = \mathbb{Q}[\beta]$ and $\beta_1, \ldots, \beta_{n}, \beta_1^{\rho}, \ldots, \beta^{\rho}_n$ are pairwise distinct.\\
(2) There is a unique $w \in \mathfrak{B}$ which is a common fixed point for $h(Y^1)$. 
\end{lem}
 
\begin{proof}
 The first part can be shown exactly as \cite[Lemma 4.12]{Sh00}, and for the second part we adapt an ideal of the proof of that lemma. Without loss of generality we can assume that the origin $0$ of $\mathfrak{B}$ is a fixed point for $h(Y^1)$ and our task is to show that it is the unique fixed point. We note that the maximal compact subgroup in $G(T)(\R)$ fixing the origin is isomorphic to $U(n)$, and hence with respect to the embedding $G(T)(\R) \hookrightarrow U(n,n)$ we have that $U(n) \hookrightarrow U(n) \times U(n)$ diagonally, i.e.
 $ a \mapsto (a, \overline{a})$. In particular we have an emdedding $h(Y^1) \hookrightarrow U(n) \hookrightarrow U(n,n)$. Assume now there is another point $z \in \mathfrak{B}$ which is a fixed point of $h(Y^1)$. Then we must have that $z = a z \overline{a}^{-1}$ for every element $\diag[a, \bar{a}] \in (U(n) \times U(n)) \cap h(Y^1)$. But for such a point we have that $a^* a = 1$ and hence  $\overline{a}^{-1} = \transpose{a}$. That is $ z = a z \transpose{a}$. Since $a \in U(n) \hookrightarrow U(n) \times U(n)$ we may diagonalise it, say with eigenvalues $\lambda_i$, $i =1, \ldots, 2n$  and hence we must have $z_{ij} = 0$ for every $\lambda_i \neq \lambda_j$. Taking $a$ to be the element obtained from $\beta$ above we have that $z$ has to be the origin.  
\end{proof}

We call a point fixed by some $h(Y^1)$ as above a CM point, and we note that this definition does not depend on the choice of the field $K$. For an example, take $T=\mathrm{diag}[\zeta\cdot 1_m,\zeta\cdot 1_r,-\zeta\cdot 1_m]$. This is the group $G'$ in section 2 and we have described its embedding into unitary group and the action on $\mathfrak{B}$ explicitly there. Let $h$ and $Y$ be as above, one easily checks that $0$ is the fixed point of $h(Y^1)$ and thus a CM point.

We now want to attach some CM periods to our CM points. We will do this by relating our definition with the notion of CM points of unitary groups. Indeed, our selection of the field $K$ allow us to view our group as a subgroup of a unitary group, and hence an embedding $\mathfrak{B} \hookrightarrow \mathcal{B}$. Our next aim is to relate the just defined CM points in $\mathfrak{B}$ with the well-studied, as in \cite{Sh00}, CM points of $\mathcal{H}$. It is here that we employ the idea of Shimura which was used in \cite{Sh67} (see also \cite[Section 7]{Sh78}) to study CM points in general Type C domains.

Let $w \in \mathfrak{B}\subset\mathcal{B}$ be a CM point fixed by $h(Y^{1}) \subset G\subset U(n,n)$. Then for a such a point we have that 
 \[
 \Lambda(\alpha,w) p(x,w) = p(x \alpha ,w),\,\,\,\,\,\, \alpha \in h(Y^1),\,\,\,x \in \C^{2n}
 \]
 where $\Lambda(\alpha,w) \in GL_{2n}(\mathbb{C})$ and $p(x,z) : \C^{2n} \times \mathcal{B} \rightarrow \mathbb{C}^{2n}$ are the map defined in \cite[4.7]{Sh00}. In this way we can obtain an embedding $Y \rightarrow End_{\mathbb{C}}(\mathbb{C}^{2n})$ by sending $\alpha\mapsto \Lambda(\alpha, w)$ where we have used the fact that $Y$   is spanned by $Y^1$ over $\mathbb{Q}$. We now extend this to an injection $h$ of $K \otimes_{\mathbb{Q}} Y \cong \mathcal{S}:=\mathcal{S}_1 \times \ldots \times \mathcal{S}_n$ into $End(\mathbb{C}^{2n})$ where $\mathcal{S}_i = K K_i$.  Indeed we set 
 \[
 h(\beta \otimes \alpha) p(x,w) = p(\beta x \alpha, w) = p(x \beta \alpha, w) = p(x \alpha \beta, w).
 \] 
 
 That is, the point $w$ can be seen as a fixed point of $\mathcal{S}^{1} \otimes_{\mathbb{Q}} \mathbb{R}$ where $\mathcal{S}^{1} =\{ s \in \mathcal{S} \,\,\,|\,\,\, s s^{\rho} = 1 \} $ with $\rho$ the involution on $\mathcal{S}$ induced by the complex conjugation on $KK_i$. Hence $w$ is a CM point in $\mathcal{B}$  defined in \cite[4.11]{Sh00} for unitary groups. In particular thanks to \cite[Lemma 4.13]{Sh00} we have that $w$ has algebraic entries as an element of $\mathfrak{B}\subset\mathcal{B}$ with the obvious $\overline{\mathbb{Q}}$-structure of every entry. 

\begin{rem}
Following \cite[Section 4]{Sh00}, let $\Omega=\{K,\Psi,L,\mathcal{T},\{u_i\}_{i=1}^s\}$ be a PEL-type and $\mathcal{F}(\Omega)$ family of polarised abelian varieties of PEL-type. The abelian varieties in $\mathcal{F}(\Omega)$ are parametrised by $\mathcal{B}$. More precisely, there is a bijection
\[
\Gamma\backslash\mathcal{B}\stackrel{\sim}\longrightarrow\mathcal{F}(\Omega), \Gamma=\{\gamma\in U(\mathcal{T}):L\gamma=L,u_i\gamma-u_i\in L\}.
\]
As in \cite{Sh63}, we can define $\Omega'=\{\mathbb{B},\Psi',L,T,\{u_i\}_{i=1}^s\}$ for quaternions and $\mathcal{F}(\Omega')$ are parametrized by $\mathfrak{B}$. The natural inclusion $\mathcal{F}(\Omega')\to\mathcal{F}(\Omega)$ is compatible with $\mathfrak{B}\to\mathcal{B}$. Moreover, similar to \cite{G84,Sh67} we actually have an embedding of canonical models between $\Gamma\backslash\mathfrak{B}$ and $\Gamma'\backslash\mathcal{B}$ for certain congruence subgroups $\Gamma,\Gamma'$. 
\end{rem}

As we have remarked, CM points for unitary groups have been extensively studied in \cite[Chapter II]{Sh00}. We recall some of their properties. For  $\alpha \in \mathcal{S}^1$ we put $\psi(\alpha) := \lambda(h(\alpha),w) \in GL_n(\mathbb{C})$, $\phi(\alpha):= \mu(h(\alpha),w) \in GL_n(\mathbb{C})$, and $\Phi(\alpha) = \diag[ \psi(\alpha), \phi(\alpha)] \in GL_{2n}(\mathbb{C})$. We can then find $B,C \in GL_{n}(\overline{\mathbb{Q}})$ (see \cite[page 78]{Sh00}) such that
 \[
 B \psi(\alpha) B^{-1} = \diag[\psi_1(\alpha), \ldots, \psi_n(\alpha)],\,\,\,\,\,C \phi(\alpha) C^{-1} = \diag[\phi_1(\alpha), \ldots, \phi_n(\alpha)],\,\,\,\,\,\,\forall \alpha \in \mathcal{S},
 \]  
 for some ring homomorphism $\phi_i, \psi_i : \mathcal{S} \rightarrow \mathbb{C}$, where we have $\mathbb{Q}$-linearly extended $\psi$ and $\phi$,  from $\mathcal{S}^1$ to $\mathcal{S}$.
 We set 
 \[
 \mathfrak{p}_{\infty} (w) := C^{-1} \diag[p_{\mathcal{S}}(\phi_1, \Phi) , \ldots, p_{\mathcal{S}}(\phi_n,\Phi)] C \in GL_n(\mathbb{C}),
\]
\[
\mathfrak{p}_{\infty \rho} (w) := B^{-1} \diag[p_{\mathcal{S}}(\psi_1, \Phi) , \ldots, p_{\mathcal{S}}(\psi_n,\Phi)] B \in GL_n(\mathbb{C}),
 \]
 where the CM-periods $p_{\mathcal{S}}(\psi_i,\Phi) \in \mathbb{C}^{\times}$ and $p_{\mathcal{S}}(\phi_i,\Phi) \in \mathbb{C}^{\times}$ are defined as in \cite[page 78]{Sh00}.  Actually we should remark here that the periods $p_{\mathcal{S}}(\psi_i,\Phi)$, $p_{\mathcal{S}}(\phi_i,\Phi)$ are uniquely determined up to elements in $\overline{\mathbb{Q}}^{\times}$, but this is sufficient for our applications.
 
We now use the fact that $w\in\mathfrak{B}\subset\mathcal{B}$ is a CM point for both $(Y,h)$ and also for $(\mathcal{S},h)$. Note that $\psi(\alpha) = \phi(\alpha)$ for $\alpha \in Y^1  \subset \mathcal{S}^1$. Indeed, for $\alpha\in G(\R)$ we have that (see \cite[eq. 2.18.9]{Sh67}), $\left[ \begin{matrix} a &  b\\ c& d \end{matrix} \right] = \left[ \begin{matrix} \overline{d} &  -\overline{c}\\ -\overline{b}& \overline{a} \end{matrix} \right]$ and hence in particular we have that $\lambda(\alpha,z) = \mu(\alpha,z)$ since $\transpose{z} = - z$. In particular the values $\psi(\alpha) = \phi(\alpha) = \lambda(\alpha,w) = \mu(\alpha,w)$ for $\alpha \in Y^1$, that is the restrictions of $\phi$ and $\psi$ to $Y^1$ are independent of the choice of the field $K$. Furthermore we note that $\psi(\alpha) = \phi(\alpha)$ for all $\alpha \in K$ with $\alpha \overline{\alpha} =1$ seen as elements of $U(n,n)$ i.e. $\alpha 1_{2n} \in U(n,n)$.

In the following lemma we use the notation $I_Y, J_Y, J_{\mathcal{S}_j}$ as defined in \cite[page 77]{Sh00}.

 \begin{lem}
 With notation as above, for all $1 \leq i \leq n$, we have that
 \[
 p_{\mathcal{S}}(\psi_i,\Phi)  = p_Y(\text{Res}_{\mathcal{S}/Y}(\psi_i), \Phi') =p_Y(\text{Res}_{\mathcal{S}/Y}(\phi_i), \Phi')= p_{\mathcal{S}}(\phi_i,\Phi) ,
 \]
 where $\Phi' = \text{Res}_{\mathcal{S}/Y} \phi = \text{Res}_{\mathcal{S}/Y} \psi \in I_Y$. 
 \end{lem}
 \begin{proof}
 Let us write  $\Phi = \sum_{j=1}^{n} \Phi_j$ with $\Phi_j \in I_{\mathcal{S}_j}$ and $\Phi' = \sum_{j=1}^n \Phi_j'$, with $\Phi_j' \in I_{K_j}$. Then we  
 have that $\Phi_j = \text{Inf}_{\mathcal{S}_j/K_j} (\Phi'_j)$. Indeed first we observe that $\Psi = \sum_{j=1}^{n} \text{Res}_{\mathcal{S}_j/K} \Phi_j \in I_K$ (see \cite[page 85]{Sh00}), where $\Psi$ as in the Remark 5.2 above. Moreover we know that $\Phi = \phi + \psi$ with $\phi,\psi \in I_{\mathcal{S}}$ as above and
we have seen that $\psi = \overline{\phi}$ when restricted to $K$ via $K \hookrightarrow Y \otimes_{\mathbb{Q}} K = \mathcal{S}$. But on the other hand we have seen that $\psi = \phi$ when restricted to $Y$, from which we obtain that $\Phi_j = \Phi_j' \otimes \tau + \Phi_j' \otimes \overline{\tau}$,  where $\tau$ a fixed embedding of $K \hookrightarrow \mathbb{C}$ ( i.e. a CM type for $K$). Since $\mathcal{S}_j = K_j \otimes_{\mathbb{Q}} K$ the claim that $\Phi_j = \text{Inf}_{\mathcal{S}_j/K_j} (\Phi'_j)$ now follows.
 
 The statement of the Lemma is now obtained from the inflation-restriction properties of the periods (see \cite[page 84]{Sh00}):
 \[
 p_{\mathcal{S}}(\psi_i,\Phi)  = \prod_{j=1}^n p_{\mathcal{S}_j}(\psi_{ij},\Phi_j) = \prod_{j=1}^n p_{K_j}(\text{Res}_{\mathcal{S}_j/K_j}(\psi_{ij}), \Phi'_j)=p_Y(\text{Res}_{\mathcal{S}/Y}(\psi_i), \Phi'),
 \] 
where $\psi_{ij} \in J_{\mathcal{S}_j}$ induced by $\psi_i \in J_{\mathcal{S}} = \bigcup_{j=1}^n J_{\mathcal{S}_j}$. Similarly follows also the other equality.
 \end{proof}

 The above lemma shows that we have $\mathfrak{p}_{\infty} (w)  = \mathfrak{p}_{\infty \rho} (w)$ for $w \in \mathfrak{B}$ and they are independent of the choice of the imaginary quadratic field $K$ we chose above (and hence of the embedding to the unitary group). We then simply define $\mathfrak{p}(w)=\mathfrak{p}_{\infty}(w)=\mathfrak{p}_{\infty\rho}(w)$ for the period attached to CM point $w\in\mathfrak{B}$. By \cite[Proposition 11.5]{Sh00} and the definition of periods we immediately have
 \begin{enumerate}
 \item The coset $\mathfrak{p}(w) GL_n(\overline{\mathbb{Q}})$ is determined by the point $w \in \mathfrak{B}$ independently of the embedding $(Y,h)$ chosen above,
 \item $\mathfrak{p} (\gamma w) GL_n(\overline{\mathbb{Q}}) = \lambda(\gamma, w) \mathfrak{p}(w) GL_n(\overline{\mathbb{Q}})$ for all $\gamma \in G(\Q)$.
 \end{enumerate}

\begin{rem}
 Even though the definition of a CM point in $\mathfrak{B}$ given above is enough for our applications, we mention here that there is a more general definition as follows. We may take $Y$ above as $Y = M_{n_1}(K_1) \times \ldots \times M_{n_s}(K_s)$ with $K_i$ CM fields and the condition that $n = \sum_{i=1}^s n_i [K_i : \mathbb{Q}]$ and assume there exists an embedding $h : Y^1 \rightarrow G(T)$ where $Y^1 := \{ y \in Y \,\,\,|\,\,\,y y^\rho = 1\}$ with the involution on $Y$ induced by complex conjugation and transpose. Then one can show as above that $h(Y^1 \otimes_{\mathbb{Q}} \mathbb{R})$ has a uniques fixed point $w \in \mathfrak{B}$. Picking as before an imaginary quadratic field $K$ disjoint from all $K_i$ we can see that the point $w \in \mathfrak{B} \hookrightarrow \mathcal{B}$ corresponds to an abelian variety $A_w$ with endomorhism ring equal to $Y \otimes_{\mathbb{Q}} K$. In particular we have that $A_w$ is isogenous to $A_1^{n_1} \times \ldots \times A_s^{n_s}$ where the abelian variety $A_i$ has CM by the field $\mathcal{S}_i:= K K_i$.
 \end{rem}

\subsection{Algebraic modular forms}

We keep the notation from before.  In particular we write $G$ for $G(T)$ and we have an embedding $\mathfrak{i}:G \to U(\mathcal{T})$ as above. For the following considerations we need to augment our definition of modular forms from scalar valued to vector valued.

We start with a $\overline{\Q}$-rational representation $\omega:\GL_n(\C)\to\GL(V)$. Given a function $f:\mathfrak{B}\to V$ and $g\in G$ define $(f|_{\omega}g)(z)=\omega(\lambda(g,z))^{-1}f(gz)$. For a congruence subgroup $\Gamma$, the space of modular forms $M_{\omega}(\Gamma)$ consists of holomorphic function with the property $f|_{\omega}\gamma=f$ for all $\gamma\in\Gamma$. Put $M_{\omega}=\bigcup M_{\omega}(\Gamma)$ where the union is over all congruence subgroups, and
\[
\mathfrak{A}_{\omega}=\bigcup_e\{g^{-1}f:f\in M_{\tau_e},0\neq g\in M_e\},
\]
\[
\mathfrak{A}_{\omega}(\Gamma)=\{h\in\mathfrak{A}_{\omega}:h|_{\omega}\gamma=h\text{ for }\gamma\in\Gamma\},
\]
where $e$ runs over $\Z$ and $\tau_e$ denotes the representation defined by $\tau_e(x)=\det(x)^e\omega(x)$.

\begin{defn}
Let $\mathcal{W}$ be a set of CM points which is dense in $\mathfrak{B}$. Put $\mathfrak{P}_{\omega}(w)=\omega(\mathfrak{p}(w))$ for $w\in\mathcal{W}$.\\
(1) An element $f\in\mathfrak{A}_{\omega}$ is called algebraic, denoted as $f\in\mathfrak{A}_{\omega}(\overline{\Q})$, if $\mathfrak{P}_{\omega}(w)^{-1}f(w)$ is $\overline{\Q}$-rational for every $w\in\mathcal{W}$ where $f$ is finite.\\
(2) We set $M_{\omega}(\overline{\Q}):=M_{\omega}\cap\mathfrak{A}_{\omega}(\overline{\Q})$, and $M_{\omega}(\Gamma,\overline{\Q}):=M_{\omega}(\Gamma)\cap M_{\omega}(\overline{\Q})$.
\end{defn}

We can compare the definition for our group with the unitary group. Let $\omega:\mathrm{GL}_n(\C)\times\mathrm{GL}_n(\C)\to\mathrm{GL}(V)$ in be $\overline{\Q}$-rational representation. Denote $\mathcal{A}_{\omega},\mathcal{A}_{\omega}(\Gamma)$ for modular function spaces for unitary group as in \cite[5.3]{Sh00}. The composition of $\omega$ with the diagonal embedding $\mathrm{GL}_n\to\mathrm{GL}_n\times\mathrm{GL}_n$ gives a representation $\omega:\mathrm{GL}_n(\C)\to\mathrm{GL}(V)$. Clearly if $f\in\mathcal{A}_{\omega}$ then its pullback $f\circ\iota\in\mathfrak{A}_{\omega}$ is a quaternionic modular form. Moreover, $f\in\mathcal{A}_{\omega}(\overline{\Q})$ and $f\circ\iota$ is finite then the pullback $f\circ\iota\in\mathfrak{A}_{\omega}(\overline{\Q})$. 

Even though we have provided a definition of algebraicity for modular forms on the bounded domain $\mathfrak{B}$ we can transfer it also to the other realisation of the symmetric spaces discussed in section 2. Indeed, with the notation of section 2.2 suppose we are given two of these realisations $(\mathfrak{i}_1,\Phi_1,H_1,K_1,\mathcal{G}_1)$ and $(\mathfrak{i}_2,\Phi_2,H_2,K_2,\mathcal{G}_2)$, with $\mathfrak{i}_1=\mathfrak{i}_2$ both induced from an algebraic embedding $M_n(\mathbb{B})\to M_{2n}(K)$, where $K$ is an imaginary quadratic field which splits $\mathbb{B}$. In particular the matrix $R$ in Equation (\ref{The matrix R}) has algebraic entries and hence we obtain that the bijective map $\rho: \mathcal{H}_1\to\mathcal{H}_2$ as defined there is algebraic, in the sense that maps a algebraic points of $\mathcal{H}_1$ to algebraic points of $\mathcal{H}_2$. We also conclude from this that $\mu(z)$, as defined in the same equation is algebraic if $z$ is. In particular given any realisation $\mathcal{H}$ there is a bijection $\rho : \mathcal{H} \rightarrow \mathfrak{B} $. We define the CM points on $\mathcal{H}$ to be the inverse image with respect to $\rho$ of the CM points of $\mathfrak{B}$.

As we have discussed every vector valued modular form $g : \mathcal{H} \rightarrow V$,  corresponds uniquely to a modular form $f:\mathfrak{B} \to V$  by the rule $g(z)=\omega(\mu(z))^{-1}f(\rho(z))$. So it is enough to now observe that if $w$ is a CM point of $\mathcal{H}$, which by definition means $\rho(w)$ is a CM point of $\mathfrak{B}$, and we have established above $\mu(w) \in GL_n(\overline{\mathbb{Q}})$. Hence we can use the same periods  $\mathfrak{P}_{\omega}(w)$ for both $f$ and $g$.

In particular, the algebraicity as defined for bounded domains can be transferred to the unbounded domain $\mathfrak{Z}$. Let now $f:\mathfrak{Z}\to\C$ be a weight $k$ modular form defined in section 3, we can take its Fourier-Jacobi expansion. We denote the Fourier-Jacobi coefficients by $c(\tau,f;v,w)$. When $r=0,n=2m$, we simply denote it by $c(\tau,f)$.
 
\begin{prop}
(1) For congruence subgroup $\Gamma$ we have $M_{k}(\Gamma)=M_{k}(\Gamma,\overline{\Q})\otimes_{\overline{\Q}}\C$.\\
(2) For every $f\in M_k$ and $\sigma\in\mathrm{Aut}(\C/\overline{\Q})$ we have $c(\tau,f^{\sigma};0,w)=c(\tau,f;0,w)$, for all $\tau$ and $w$.\\
(3) Let $r=0,n=2m$, (the Tube Domain case) and $f\in M_k$ then $f\in M_k(\overline{\Q})$ if and only if $c(\tau,f)\in\overline{\Q}$ for all $\tau$;\\
(4) For congruence subgroup $\Gamma$ and $\sigma\in\mathrm{Aut}(\C/\overline{\Q})$ we have $S_k(\Gamma)^{\sigma}=S_k(\Gamma)$ and $S_k(\Gamma)=S_k(\Gamma,\overline{\Q})\otimes_{\overline{\Q}}\C$.
\end{prop}

\begin{proof}
This can be proved similarly as \cite[Proposition 11.11, 11.15, 26.8]{Sh00}. See also \cite[Proposition 7.2]{M} for (1) and \cite{G84} for (2),(3). 

We briefly explain the proof for (3). Let $r=0,n=2m$ and $f\in M_k(\Gamma)$ for a congruence subgroup $\Gamma$. Let $V$ be the model of $\Gamma\backslash\mathfrak{Z}$ defined over $\overline{\Q}$ then $\mathfrak{A}_0(\Gamma)$ can be identified with the function field of $V$. By the same method in \cite[Sections 6,7]{Sh00}, one can show that $g\in\mathfrak{A}_0(\Gamma)$ if and only if $g$ has algebraic Fourier coefficients. We will reduce our problem for $f\in M_k$ to $\mathfrak{A}_0$ similarly to what is done in the proof of \cite[Proposition 11.11]{Sh00}.

 Let $\mathcal{W}$ be a dense subset of CM points in $\mathfrak{Z}$. We first assume that $\det(\mathfrak{p}(w))^{-k}f(w) \in \overline{\Q}$ for all $w\in\mathcal{W}$ where $f$ is finite. Note that there exists a function $U\in\mathfrak{A}_k(\overline{\Q})$ on $\mathfrak{Z}$ holomorphic in $w$ with $\det(U)(w)\neq 0$. Indeed, denote $\mathcal{H}$ for the unbounded realization of $\mathcal{B}$ via Cayley transform, we can simply put $U(z)=R(z)$ for $z\in\mathfrak{Z}\hookrightarrow\mathcal{H}$ with $R$ the function in \cite[Proposition 9.11]{Sh00}. We set $g:=\det(U)^{-k}f$, and note that
\[
 g(w) = \det\left(U(w)^{-1} \mathfrak{p}(w)\right)^k \det(\mathfrak{p}(w))^{-k} f(w).
\]
But now we have that $U(w)^{-1}\mathfrak{p}(w)$ is $\overline{\mathbb{Q}}$ -rational since this holds for the function $R$ in unitary case. That is $g(w)$ is $\overline{\mathbb{Q}}$-rational for every CM point $w$ where $g$ is finite thus $g\in\mathfrak{A}_0(\overline{\Q})$. Since $ f = \det(U)^{-k}g$ we obtain that $f$ also has algebraic Fourier expansion.

For the other direction, we keep the same notation. If $f \in M_k$ has algebraic Fourier expansion then $g \in \mathfrak{A}_{0}(\overline{\mathbb{Q}})$. For every CM point $w$ we may choose the function $U$ above such that $U$ is finite at $w$ and $U(w)$ is invertible. If $f$ is finite at $w$ then so is $g$ and $g(w)$ is $\overline{\mathbb{Q}}$ rational. The equality of $g(w)$ as above then shows that $\det(\mathfrak{p}(w))^{-k} f(w)$ is $\overline{\mathbb{Q}}$-rational, and hence $f\in M_k(\overline{\Q})$. 
 \end{proof}
 
We end this subsection by giving a definition for adelic modular forms. Let $\mathbf{f}\in\mathcal{M}_k$ be a weight $k$ (adelic) modular form. We say that $\mathbf{f}$ is algebraic, denoted as $\mathbf{f}\in\mathcal{M}_k(\overline{\Q})$ if for a dense subset $\mathcal{W}$ of CM points in $\mathfrak{Z}$,
$\mathfrak{P}_k(w)^{-1}\mathbf{f}(g_{\mathbf{h}}g)\in\overline{\Q}$ for all $w\in\mathcal{W}$. Since $j(g,w)\in\overline{\Q}$ for CM point $w$, this is the same as all component $f_j$ under correspondence $\mathbf{f}\leftrightarrow(f_0,...,f_h)$ are algebraic. 

\begin{prop}
(1) Let $K=K_1(\mathfrak{n})$ or $K_0(\mathfrak{n})$ then $\mathcal{M}_{k}(K)=\mathcal{M}_k(K,\overline{\Q})\otimes_{\overline{\Q}}\C,\mathcal{S}_{k}(K)=\mathcal{S}_k(K,\overline{\Q})\otimes_{\overline{\Q}}\C$.\\
(2) Let $r=0,n=2m$ and $\mathbf{f}\in\mathcal{M}_k$, then $\mathbf{f}\in\mathcal{M}_k(\overline{\Q})$ if and only if the Fourier coefficients $c(\tau,q,\mathbf{f})\in\overline{\Q}$. 
\end{prop}

Let $r=0,n=2m$ and keep the notation for Eisenstein series in previous sections. Let $\mathbf{E}_l(g,s)$ be a Siegel Eisenstein series for group $G_n$. By explicit computation of Fourier expansion, we have $\mathbf{E}_l^{\ast}(g,l)\in\mathcal{M}_l(\overline{\Q})$. Clearly we also have $\mathbf{E}_l(g,l)\in\mathcal{M}_l(\overline{\Q})$.

\subsection{Differential operators and nearly holomorphic functions}

In this subsection, we summarise some of the result of \cite{Sh94}, (see also \cite[Chapter 3]{Sh00}) on differential operators on type D domains and then apply these operators to Siegel-type Eisenstein series.
We will be working with the the bounded realisation of our symmetric space but thanks to the remark above we can transfer the definitions from one realisation to the other. We set
\[
\mathfrak{B}=\{z\in\C_n^n:\transpose{z}=- z,z^{\ast}z<1_n\},\,\,\mathfrak{T}:=\{z\in\C_n^n:\transpose{z}=-z\},\,\,\eta(z):=1-z^{\ast}z.
\]
Here $\mathfrak{T}$ is the tangent space of $\mathfrak{B}$ at the origin $0$. 

Given a positive integer $d$ and two finite-dimensional complex vector spaces $W$ and $V$, we denote by $Ml_d(W,V)$ the vector space of all $\C$-multilinear maps of $W\times...\times W$ ($d$ copies) into $V$ and $S_d(W,V)$ the vector space of all homogeneous polynomial maps of $W$ into $V$ of degree $d$. We omit the symbol $V$ if $V=\C$. Given a representation $\omega:\mathrm{GL}_n(\C)\to V$ we define a representation $\{\omega\otimes\tau^d,Ml_d(\mathfrak{T},V)\}$ by
\[
[(\omega\otimes\tau^d)(a)h](u_1,...,u_d)=\omega(a)h(\transpose{a}u_1a,..,\transpose{a}u_da),
\]
for $a\in\mathrm{GL}_n(\C),h\in Ml_d(\mathfrak{T},V),u_i\in\mathfrak{T}$. In particular taking $d=1$ and $\omega$ the trivial representation, we define the representation $\{\tau,S_1(\mathfrak{T})\}$ of $\GL_n(\C)$ by $[\tau(a)h](u)=h(\transpose{a}ua)$ for $h\in S_1(\mathfrak{T}),u\in\mathfrak{T}$.

Take an $\R$-rational basis $\{\epsilon_{\nu}\}$ of $\mathfrak{T}$ over $\C$ and for $u\in\mathfrak{T}=\sum_{\nu}u_{\nu}\epsilon_{\nu}$. For $z\in\mathfrak{B}$, write $z=\sum_{\mu}z_{\nu}\epsilon_{\nu}$. For $f\in C^{\infty}(\mathfrak{B},V)$ we define $\mathfrak{D}f,\overline{\mathfrak{D}}f,\mathfrak{C}f\in C^{\infty}(\mathfrak{B},S_1(\mathfrak{T},V))$ by 
\[
(\mathfrak{D}f)(u)=\sum_{\nu}u_{\nu}\frac{\partial f}{\partial z_{\nu}},\,\,(\overline{\mathfrak{D}}f)(u)=\sum_{\nu}u_{\nu}\frac{\partial f}{\partial \overline{z}_{\nu}},\,\,(\mathfrak{C}f)(u)=(\mathfrak{D}f)(\transpose{\eta}(z)u\eta(z)).
\]

We further define $\mathfrak{D}^df,\overline{\mathfrak{D}}^df,\mathfrak{C}^df$ by
\[
\mathfrak{D}^df=\mathfrak{DD}^{d-1}f,\,\,\overline{\mathfrak{D}}^df=\overline{\mathfrak{D}}\overline{\mathfrak{D}}^{d-1}f,\,\,\mathfrak{C}^df=\mathfrak{CC}^{d-1}f,\,\,\mathfrak{D}^0f=\overline{\mathfrak{D}}^0f=\mathfrak{C}f=f.
\]
And $\mathfrak{D}_{\omega}^df\in C^{\infty}(\mathfrak{B},S_d(\mathfrak{T},V))$ by
\[
\mathfrak{D}_{\omega}^df=(\omega\otimes\tau^d)(\eta(z))^{-1}\mathfrak{C}^d[\omega(\eta(z))f].
\]
We now recall the important fact, due to Hua, Schmid, Johnson and Shimura (see for example \cite{Sh84}), that the representation $\{\tau^d,S_d(\mathfrak{T})\}$ is the direct sum of irreducible representations and each irreducible constituent has mulitplicity one. In particular for each $\mathrm{GL}_n(\C)$-stable subspace $Z\subset S_d(\mathfrak{T})$ we can define the projection map $\phi_Z$ of $S_d(\mathfrak{T})$ onto $Z$. Define $\mathfrak{D}_{\omega}^Zf\in C^{\infty}(\mathfrak{B},Z\otimes V)$ by $\mathfrak{D}_{\omega}^Zf=\phi_Z\mathfrak{D}_{\omega}^df$.

\begin{lem} With notation as above we have,
%(1) Define a representation $\{\tau,S_1(\mathfrak{T})\}$ of $\GL_n(\C)$ by $[\tau(a)h](u)=h(\transpose{a}ua)$ for $h\in S_1(\mathfrak{T}),u\in\mathfrak{T}$. Then
\begin{enumerate}
 \item $\pi^{-1}\mathfrak{D}f\in\mathfrak{A}_{\tau}(\overline{\Q})$ for every $f\in\mathfrak{A}_0(\overline{\Q})$.\newline
\item Let $Z$ be a $\mathrm{GL}_n(\C)$-stable subspace of $S_d(\mathfrak{T})$. If $f\in\mathfrak{A}_{\omega}(\overline{\Q})$ then 
\[
\pi^{-d}\mathfrak{P}_{\omega}(w)^{-1}\mathfrak{D}_{\omega}^Zf(w)
\]
is $\overline{\Q}$-rational for any CM point $w$.

%$w\in\mathcal{W}=\{g\cdot 0:g\in G(\overline{\Q})\}$ where $f$ is finite.
\end{enumerate}
\end{lem}

\begin{proof}
The proof is same as the one in \cite[Theorem 14.5, Theorem 14.7]{Sh00} (see also \cite[Sections 5 and 6]{Sh84}). Indeed, as we have a natural inclusion $\mathfrak{B}\to\mathcal{B}$ we can reduce our problem to unitary case. For example for (1), denote $Df$ for the differential operators in unitary case. The lemma is proved for this case in \cite[Theorem 14.5]{Sh00}. Let $p$ be the complex dimension of $\mathfrak{B}$, we can take $p$ elements $g_1,...,g_p\in\mathcal{A}_0(\overline{\Q})$ such that $g_1\circ\epsilon,...,g_p\circ\epsilon$ are algebraically independent.  Put $f_j=g_j\circ\epsilon$. As shown in last section $g_j\circ\epsilon\in\mathfrak{A}_0(\overline{\Q})$ so $\partial/\partial f_1,...,\partial/\partial f_p$ are well-defined derivations of $\mathfrak{A}_0(\overline{\Q})$. For every $f\in\mathfrak{A}_0(\overline{\Q})$ we have $\mathfrak{D}f=\sum_j(\partial f/\partial f_j)\mathfrak{D}f_j$. Now $\mathfrak{D}(f_j)=(Dg_j)\circ\epsilon$ and $\pi^{-1}Dg_j$ is $\overline{\Q}$-rational. This proves our assertion.
\end{proof}

We now set $r(z):=-\eta(z)^{-1}\overline{z}$. Let $d$ be a nonnegative integer and $\{\omega,V\}$ the representation as before. A function $f\in C^{\infty}(\mathfrak{B},V)$ is called nearly holomorphic of degree $d$ if it can be written as a polynomial in $r$, of degree less than $d$, with $V$-valued holomorphic functions on $\mathfrak{B}$ as coefficients. We denote the space of such functions by $\mathfrak{N}^d(\mathfrak{B},V)$. Let $\mathfrak{N}^d_{\omega}$ be the space consisting of functions satisfying the modular properties as in $M_{\omega}$ but now replacing the holomorphic condition with nearly holomorphic. For a congruence subgroup $\Gamma$ we can similarly define the space $\mathfrak{N}^d_{\omega}(\Gamma)$. An exact same argument an in the proof of \cite[Lemma 14.3]{Sh00}, shows that this space is finite-dimensional over $\C$.

Suppose $V$ is $\overline{\Q}$-rational. A function $f\in\mathfrak{N}_{\omega}^d$ is called algebraic, denoted as $f\in\mathfrak{N}_{\omega}^d(\overline{\Q})$, if $\mathfrak{P}_{\omega}(w)^{-1}f(w)$ is $\overline{\Q}$-rational for $w\in \mathcal{W}=\{g\cdot 0:g\in G(\overline{\Q})\}$. Put $\mathfrak{N}_{\omega}^d(\Gamma,\overline{\Q})=\mathfrak{N}_{\omega}^d(\overline{\Q})\cap\mathfrak{N}_{\omega}^d(\Gamma)$. The proof of following lemma is same as the one in \cite[Theorem 14.9]{Sh00}.

\begin{lem}
Let $Z$ be an irreducible subspace of $S_p(T)$. Then $\pi^{-p}\mathfrak{D}_{\omega}^Zf\in\mathfrak{N}_{\omega\otimes\tau_Z}^{d+p}(\overline{\Q})$ for every $f\in\mathfrak{N}_{\omega}^d(\overline{\Q})$. Here $\tau_Z$ is the restriction of $\tau^p$ to $Z$. 
\end{lem}

We now extend the above definitions to adelic modular forms. Let $\mathbf{f}\in\mathcal{M}_k$ and viewing it as a function on $G(\mathbb{A}_{\mathbf{h}})\times\mathfrak{B}$ by setting $\mathbf{f}(g_{\mathbf{h}},z)=j(g_z,z_0)^k\mathbf{f}(g_{\mathbf{h}}g_z)$ with $z=g_z\cdot 0\in\mathfrak{B}$. Then $\mathfrak{D}_k\mathbf{f},\mathfrak{D}_k^Z\mathbf{f}$ is defined as applying differential operators on $z\in\mathfrak{B}$. A function $\mathbf{f}:G(\mathbb{A}_{\mathbf{h}})\times\mathfrak{B}\to\C$ is called nearly holomorphic if it is nearly holomorphic in $z\in\mathfrak{B}$. We can then define the space $\mathcal{N}_k^d$ and of nearly holomorphic modular forms as before. Similarly, we can define subspace $\mathcal{N}_k^d(\overline{\Q})$. These definitions are equivalent to all components in the correspondence $\mathbf{f}\leftrightarrow(f_1,...,f_h)$ are nearly holomorphic or algebraic nearly holomorphic. 

We now apply the differential operators to Siegel-type Eisenstein series and show that it is nearly holomorphic for certain values of $s$. We will keep the notation of section 4, and so in particular $\mathbf{E}_l(g,s)$ is the Siegel-type Eisenstein series associated to group $G_n,n=2m$ of weight $l$ and character $\chi$.

\begin{prop} \label{Nearly holomorphic Eisenstein series}
Assume $l>n-1$ and let $\mu\in\Z$ such that $n-1<\mu\leq l$. Then\\
(1) $\mathbf{E}_l(g,\mu)\in\pi^{\alpha}\mathcal{N}_l^{m(l-\mu)}(\overline{\Q})$ with $\alpha=m(l-\mu)$;\\
(2) Denote $\mathfrak{E}_l(g,s)=\Lambda_{\mathfrak{n}}(s,\chi)\mathbf{E}_l(g,s,\chi)$. Then $\mathfrak{E}_l(g,\mu)\in\pi^{\beta}\mathcal{N}_l^{m(l-\mu)}(\overline{\Q})$ with $\beta=m(l+\mu)-m(m-1)$.
\end{prop}

\begin{proof} For this we use \cite[Theorem 2D]{Sh84} which classifies the irreducible representations of $(\tau^{mp}, S_{mp}(\mathfrak{T}))$. In particular for 
for $p\in\Z$ and a weight $q$ we can define the operator $\Delta_q^p$ by $\Delta_q^p\mathbf{f}=(\mathfrak{D}_{\omega}^Z\mathbf{f})(\psi)$ with $\omega=\det^q,Z=\C\psi\subset S_{mp}(\mathfrak{T})$ and $\psi=\det^{p/2}$. Here the square root of the determinant denotes the Pfaffian of the skew-symmetric matrix. Then 
\[
\Delta_q^p\mathcal{N}_q^t(\overline{\Q})\subset\pi^{mp}\mathcal{N}_{q+p}^{t+mp}(\overline{\Q}).
\]
We have shown that $\mathbf{E}_l(g,l)\in\mathcal{M}_l(\overline{\Q})$ so $\Delta_l^p\mathbf{E}_l(g,l)\in\pi^{mp}\mathcal{N}_{l+p}^{mp}(\overline{\Q})$. Take $p=l-\mu$, then by the explicit formula in \cite[Theorem 4.3]{Sh84}, we have
\[
\Delta_{\mu}^p\mathbf{E}_{\mu}(g,\mu)=C\cdot\mathbf{E}_l(g,\mu),\,\,\,\text{with}\,\,C\in\overline{\Q}^{\times}.
\]
This concludes the proof of the proposition.
\end{proof}

\section{Main Results}

We now recall that we have established the integral representation of the L-function,
\[
L(s,\mathbf{f},\chi)\mathbf{f}(g)=c_k(s)\int_{G(\Q)\backslash G(\mathbb{A})/K_1(\mathfrak{n})K_{\infty}}\mathfrak{E}(g\times h,s)\mathbf{f}(h)\mathbf{d}h.
\]
Here $\mathfrak{E}_k(\mathfrak{g},s)=\Lambda_{\mathfrak{n}}(s,\chi)\mathbf{E}(\mathfrak{g},s)$; $\mathbf{E}(\mathfrak{g},s)=E_k(\mathfrak{g}\sigma^{-1},s)$ for $\mathfrak{g}\in G_N(\mathbb{A})$ and $\sigma=1$ if $v\nmid\mathfrak{n}$, $\sigma=\tilde{\tau}_m$ if $v|\mathfrak{n}$, where $E_k(\mathfrak{g},s)$ is the Siegel-type Eisenstein series defined on $G_N$ of  weight $k$ and $N=2n$. 

We first prove a Lemma which is the analogue of \cite[Lemma 26.12]{Sh00} in our setting.

\begin{lem} \label{decomposition}
Let $\mathbf{f}\in\mathcal{N}^d_k(\overline{\Q})$ be an algebraic nearly holomorphic form associated to group $G_N$. Then there exist $\mathbf{g}_j,\mathbf{h}_j\in\mathcal{N}_k^d(\overline{\Q})$ associated to group $G_n$ such that
\[
\mathbf{f}(g\times h)=\sum_{j=1}^e\mathbf{g}_j(g)\overline{\mathbf{h}_j(h)}.
\]
\end{lem}

\begin{proof}
Write
\[
\mathbf{f}(g\times h)=j(g_{\infty}\times h_{\infty},z_0\times z_0)^{-k}\mathbf{f}(g_{\mathbf{h}}\times h_{\mathbf{h}},z\times w)
\]
with $z=g_{\infty}z_0,w=h_{\infty}z_0\in\mathfrak{Z}_{m,r}$. By definition, $f(z,w):=\mathbf{f}(g_{\mathbf{h}}\times h_{\mathbf{h}},z\times w)$ is nearly holomorphic in $z\times w$. Similarly to the proof of \cite[Lemma 26.12]{Sh00}, one can show that it is also nearly holomorphic in $z$ and $\overline{f(z,w)}$ is nearly holomorphic in $w$. Therefore $\mathbf{f}\in\mathcal{N}_k^d(\overline{\Q})$ (resp. $\overline{\mathbf{f}}\in\mathcal{N}_k^d(\overline{\Q})$) as a function in $g$ or $h$.

Let $\{\mathbf{g}_j\}_{j=1}^e$ be a $\overline{\Q}$-rational basis of $\mathcal{N}_k^d(\overline{\Q})$. For each fixed $h$ we have $\mathbf{f}(g\times h)=\sum_{j=1}^e\mathbf{g}_j(g)\overline{\mathbf{h}_j(h)}$ with $\mathbf{h}_j(h)\in\C$. Since $\mathbf{g}_j$ are linearly independent we can find $e$ points $g_1,...,g_e$ such that $\det(\mathbf{g}_j(z_k))_{j,k=1}^e\neq 0$. Solving the linear equations $\overline{\mathbf{f}(g,h)}=\sum_{j=1}^e\overline{\mathbf{g}_j(z_k)}\mathbf{h}_j(h)$ we find functions $\mathbf{h}_j\in\mathcal{N}_k^d$. 

It suffices to prove $\{\mathbf{h}_j\}$ are algebraic. Since $\mathcal{W}=\{g\cdot z_0:g\in G(\overline{\Q})\}$ is a dense subset of $\mathfrak{Z}_{m,r}$, we can take $g_j$ such that $g_jz_0\in\mathcal{W}$. We easily calculate the period $\mathfrak{P}_k(z_0\times z_0)=\mathfrak{P}_k(z_0)\mathfrak{P}_k(z_0)$ hence
\[
\mathfrak{P}_k(z_0\times z_0)^{-1}\overline{\mathbf{f}(g_{\mathbf{h}}\times h_{\mathbf{h}})}=\sum_{j=1}^e\mathfrak{P}_k(z_0)^{-1}\overline{\mathbf{g}_j(g_{\mathbf{h}})}\mathfrak{P}_k(z_0)^{-1}h_j(\mathbf{h}).
\]
By algebraicity of $\mathbf{f},\mathbf{g}_j$ we have $\mathfrak{P}_k(z_0)^{-1}h_j(\mathbf{h})\in\overline{\Q}$ and thus for all $w=h_{\infty}z_0\in\mathcal{W}$ we have $\mathfrak{P}_k(w)^{-1}h_j(h)\in\overline{\Q}$. Hence $\mathbf{h}_j\in\mathcal{N}_k^d(\overline{\Q})$ which completes the proof.
\end{proof}

Before stating the main theorem we need to establish one more result. Namely we show that
\begin{prop} Assume $k > 2n-1$ and let $\mu\in\Z$ such that $2n-1<\mu\leq k$. Then there exists a function $\mathbf{T}(g,h)$ with $\overline{\mathbf{T}(g,h)} \in \mathcal{N}^{m(k-\mu)}_k(\overline{\Q}) \times \mathcal{M}_k(\overline{\mathbb{Q}}) $, such that 
\[
\langle \mathbf{T}(g, h),\mathbf{f}(h)\rangle = \langle \mathfrak{E}(g \times h,\mu),\mathbf{f}(h)\rangle.
\]
\end{prop}

\begin{proof} This is the analogue of Lemma 29.3 proved in \cite{Sh00} in the unitary case. Actually it is even simpler in our case since we do not need to involve some more complicated differential operators needed in the unitary case. Here we simply indicate some changes to the proof in \cite{Sh00} to cover our case. We follow the notation of the Appendix A8 in \cite{Sh00} and write $\mathfrak{g}$ for the real Lie algebra of $\mathbf{G}:=G(\mathbb{R})$. We then have the familiar decomposition of the complexification $\mathfrak{g}_{\mathbb{C}} = \mathfrak{t}_{\mathbb{C}} \oplus \mathfrak{p}_{+} \oplus \mathfrak{p}_{-}$ where $\mathfrak{t}$ is the Lie algebra of the fixed maximal compact subgroup $\mathbf{K} \cong U(n)$. Finally we write $\mathfrak{U}$ for the universal enveloping Lie algebra of $\mathfrak{g}_{\mathbb{C}}$, and $\mathbf{K}^c = GL_n(\mathbb{C})$ for the complexification of $\mathbf{K}$.

Given a representation $(\rho,V)$ of $\mathbf{K}^{c}$ we write $C^{\infty}(\rho)$ for the functions $f \in C^{\infty}(\mathbf{G},V)$ such that $f(xk^{-1}) = \rho(k) f(x)$ for all $k \in \mathbf{K} \subset \mathbf{K}^c$, and $x \in \mathbf{G}$. As in \cite{Sh00} there is a bijection between $C^{\infty}(\mathfrak{Z},V)$ and $C^{\infty}(\rho)$ which we denote by $f \mapsto f^{\rho}$. We also write $H(\rho)$ for the functions in $C^{\infty}(\rho)$ such that $Y f = 0$ for all $Y \in \mathfrak{p}_{-}$. These functions correspond to holomorphic functions in $C^{\infty}(\mathfrak{Z},V)$. 

Recall (see \cite{Sh00}) that a $\mathfrak{U}$-module $\mathcal{Y}$ is called unitarizable if there exists a positive definite hermitian form $\{\,,\,\} : \mathcal{Y} \times \mathcal{Y} \rightarrow \mathbb{C}$ such that $\{Xg,h\} =-\{g,Xh\}$ for every $g,h \in \mathcal{Y}$ and $X \in \mathfrak{g}$.  

Let us now take $\rho = \det^k$ for some $k > 2m-1$. Then we have that for any nonzero $f \in H(\rho)$ the $\mathfrak{U}$-module structure of $\mathfrak{U}f$ depends only on the weight $k$, and that such a module is unitarisable. This follows exactly as in \cite[Theorem A8.4]{Sh00} where the cases of unitary and symplectic groups are considered. What needs to be explained is the bound on the weight $k$. For this there are two remarks that one needs to make: first that the bound follows from the fact that in our Type D setting the function $\psi_{Z}$ is given by (see \cite{Sh84}),
\[
\psi_Z(s) = \prod_{h=1}^m \prod_{i=1}^{r_{2h}} \left( s-i+2h-1\right)
\]
For the notation we refer to \cite{Sh84} since what is important in the proof is the fact that $\psi_Z(-k) \neq 0$ which is satisfied for the selected bound on $k$. The other remark is the existence of a nonzero $g \in H(\det)$ and a discrete subgroup $\Gamma$ of $\mathbf{G}$ such that $\Gamma \setminus \mathbf{G}$ is compact and $f(\gamma x) =f(x)$ for all $\gamma \in \Gamma$. This is the analogue of \cite[Lemma A8.5]{Sh00} which covers the unitary case. But again the existence of such a $g$ and a discrete $\Gamma$ can be derived from the existence of such elements in the unitary case, say $\widetilde{\Gamma}$ and $\widetilde{g}$ (the content of Lemma A8.5 and the natural closed embedding $\mathbf{G} \hookrightarrow U(n,n)$. In particular we may take $\Gamma:= \widetilde{\Gamma} \cap \mathbf{G}$ and $g$ as the restriction of $\widetilde{g}$ to $\mathbf{G}$.

The importance of considering $\mathfrak{U}$-module structures which are unitarizable becomes clear from the following result on holomorphic projection. Namely, if we still write $\rho = \det^k$, and consider an $f \in N_{k}^d(\overline{\mathbb{Q}})$ for any $d \in \mathbb{N}$ such that $\mathfrak{U} f^\rho$ is unitarizable, then there exists an element $q \in M_k(\overline{\mathbb{Q}})$ such that $\langle f, h\rangle = \langle q, h\rangle$ for all $h \in S_{k}$. This is a rather general result and can be obtained exactly in the same way as it is done in \cite[Lemma A8.7]{Sh00} in the unitary and symplectic case with little changes.

We can now complete the proof of the proposition. First we note that unitarizable $\mathfrak{U}$-modules behave well with respect to the doubling mapping. Let us write $\mathbf{G}_{i}$ with $i=1,2$ for groups of type similar to $\mathbf{G}$, and we insert the index $i$ to all notations.
Assume we have a doubling embedding $\mathbf{G}_1 \times \mathbf{G}_1 \hookrightarrow \mathbf{G}_2$ is of the kind considered in this paper, and we write $\Delta$ for the corresponding embedding of symmetric spaces. Then if $f \in H_2(\rho)$ such that the $\mathfrak{U}_2$-module $\mathfrak{U}_2 f$ is unitarizable then the $\mathfrak{U}_1 \times \mathfrak{U}_1$ module $\Delta^* \mathfrak{U}f$ is unitarizable with respect to both variables.  This follows similar to \cite[Lemma A8.11]{Sh00}. Hence in order to complete the proof it is enough to show that the Eisenstein series $\mathfrak{E}_k(g,\mu)$ belongs to a unitarizable $\mathfrak{U}$-module, since then when we pull it back with respect to the diagonal embedding we can keep the one variable constant (the variable $g$ in the statement of the proposition) and take the holomorphic projection with respect to the other.  This final claim follows from the fact shown in Proposition \ref{Nearly holomorphic Eisenstein series} that the Eisenstein series $\mathfrak{E}_k(g,\mu)$ on $G_N$ with $N=2n$ are obtained from holomorphic ones of weight $l \geq N-1 = 2n-1$ by applying the Shimura-Maass operators $\Delta_l^p$. But these operators are well known to be obtained as operators of the universal enveloping algebra. Indeed this is shown for the symplectic and unitary case in the first few lines of \cite[A8.8]{Sh00} and more generally in \cite{Harris81}. 
\end{proof}

We can now prove the Theorem on the algebraicity of the $L$-values (we remind the reader here of the Remark \ref{Remark on the conductor} made in the introduction).

\begin{thm} \label{Main Theorem}
Let $\mathbf{f}\in\mathcal{S}_k(K_1(\mathfrak{n}), \overline{\mathbb{Q}})$ be an eigenform with $k>2n-1$, and let $\chi$ be a Dirichlet character whose conductor divides the ideal $\mathfrak{n}$. Let $\mu\in\Z$ such that $2n-1<\mu\leq k$, then
\[
\frac{L(\mu,\mathbf{f},\chi)}{\pi^{n(k+\mu)-\frac{3}{2}n(n-1)}\langle\mathbf{f,f}\rangle}\in\overline{\Q}.
\]
\end{thm}

\begin{comment}

\begin{rem} \label{Remark on the conductor}
Before giving the proof we remark that the restriction on the conductor of Dirichlet character is not necessary. Indeed, for $\mathbf{f}\in\mathcal{S}_k(K_1(\mathfrak{n}),\overline{\Q})$ and $\chi$ of conductor $\mathfrak{m}$ we can select $\mathfrak{n}'=\mathfrak{nm}$ instead of $\mathfrak{n}$ since $\mathcal{S}_k(K_1(\mathfrak{n}), \overline{\mathbb{Q}}) \subset \mathcal{S}_k(K_1(\mathfrak{n}'), \overline{\mathbb{Q}})$.
\end{rem}

\begin{rem}
Before giving the proof we remark that we can always select $\mathfrak{n}$ small enough so that the condition on the conductor of the character is satisfied. That is, if $\mathbf{f} \in  \mathcal{S}_k(K_1(\mathfrak{n}_1), \overline{\mathbb{Q}})$ and $\chi$ has conductor $\mathfrak{m}$ then we can select $\mathfrak{n}:= \mathfrak{n}_1\mathfrak{m}$ since $\mathcal{S}_k(K_1(\mathfrak{n}_1), \overline{\mathbb{Q}}) \subset \mathcal{S}_k(K_1(\mathfrak{n}), \overline{\mathbb{Q}})$. 
\end{rem}
\end{comment}

\begin{proof}
We prove the theorem following an idea used in the proof of \cite[Theorem 29.5]{Sh00}, which allows us to cover also the non-split (i.e. non-tube) case. By the above Proposition we can replace $\mathfrak{E}(g\times h,\mu)$ by $\overline{\mathbf{T}(g,h)}$ holomorphic in $h$ such that the integral can be rewritten as
\[
c_k(\mu)L(\mu,\mathbf{f},\chi)\mathbf{f}(g)=\langle \mathbf{T}(g,h),\mathbf{f}(h)\rangle.
\]
By Lemma \ref{decomposition} we have
\[
\pi^{-\beta}\mathbf{T}(g,h)=\sum_{j=1}^e\overline{\mathbf{g}_j(g)}\mathbf{h}_j(h),
\]
with $\mathbf{g}_j\in\mathcal{N}_k^{n(k-\mu)}(\overline{\Q}),\mathbf{h}_j\in\mathcal{M}_k(\overline{\Q})$, and $\beta = n(k+\mu)-n(n-1)$. We note here that indeed $\mathbf{h}_j\in\mathcal{M}_k(\overline{\Q})$, as one can see in the proof of Lemma \ref{decomposition} that the analytic properties of the $\mathbf{h}_j$'s follow from that of the restricted function on the $h$ variable since they are obtained as the solutions of a linear system where the ``constant'' vector consists of holomorphic functions.

Then above equation can be written as
\[
\frac{c_k(\mu)L(\mu,\mathbf{f},\chi)}{\pi^{\beta}}\mathbf{f}(g)=\sum_{j=1}^e\langle\mathbf{h}_j,\mathbf{f}\rangle\cdot\mathbf{g}_j(g).
\]

Since we are assuming $k > 2n-1$ we may apply \cite[Corollary 2.4.6]{Harris} and write $\mathcal{M}_k(\overline{\mathbb{Q}})=\mathcal{S}_k(\overline{\mathbb{Q}}) \oplus\mathcal{E}_k(\overline{\mathbb{Q}})$ as a direct sum of space of algebraic cusp form and the space of algebraic Eisenstein series. In particular we can find $\mathbf{h}_j'\in\mathcal{S}_k(\overline{\Q})$ such that $\langle\mathbf{h}_j,\mathbf{f}\rangle=\langle\mathbf{h}_j',\mathbf{f}\rangle$. Let $w=g_{\infty}z_0$ be a CM point with period $\mathfrak{P}_k(w)$ then by definition $\mathfrak{P}_k(w)^{-1}\mathbf{g}_j(g_{\mathbf{h}}g_{\infty})\in\overline{\Q}$. Therefore at $g={g}_{\mathbf{h}}g_{\infty}$ we can further find some $\mathbf{h}''\in\mathcal{S}_k(K_1(\mathfrak{n}),\overline{\Q})$ such that
\[
\frac{c_k(\mu)L(\mu,\mathbf{f},\chi)}{\pi^{\beta}}\mathbf{f}(g_{\mathbf{h}}\cdot g_{\infty})=\mathfrak{P}_k(w)\langle\mathbf{h}'',\mathbf{f}\rangle.
\]
Denote
\[
\mathcal{V}=\{\mathbf{f}\in\mathcal{S}_{k}(K_1(\mathfrak{n})):\mathbf{f}|T_{\xi}=\lambda(\xi)\mathbf{f}\},\mathcal{V}(\overline{\Q})=\mathcal{V}\cap\mathcal{S}_k(K_1(\mathfrak{n}))
\]
for the space consisting of eigenforms with same eigenvalues as $\mathbf{f}$. Since $\mathcal{S}_k(K_1(\mathfrak{n}))=\mathcal{S}_k(K_1(\mathfrak{n}),\overline{\Q})\otimes_{\overline{\Q}}\C$ and $\mathcal{S}_k(K_1(\mathfrak{n}),\overline{\Q})$ is stable under the action of Hecke operators, we obtain that the eigenvalues $\lambda(\xi)\in\overline{\Q}$. Hence we have $\mathcal{V}=\mathcal{V}(\overline{\Q})\otimes_{\overline{\Q}}\C$. We may now write $\mathcal{S}_k(K_1(\mathfrak{n}),\overline{\Q}))=\mathcal{V}(\overline{\Q})\oplus\mathcal{U}$ for some $\overline{\Q}$-rational vector space $\mathcal{U}$ (compare with the first few lines of the proof of \cite[Theorem 28.5]{Sh00}). With $w=g_{\infty}z_0$ as above, let $\mathbf{h}_w$ be the projection of $\mathbf{h}''$ to $\mathcal{V}(\overline{\Q})$, then for all $\mathbf{f}\in\mathcal{V}(\overline{\Q})$ and any CM point $w=g_{\infty}z_0$ we have
\[
\frac{c_k(\mu)L(\mu,\mathbf{f},\chi)}{\pi^{\beta}}\frac{\mathbf{f}(g_{\mathbf{h}}\cdot g_{\infty})}{\mathfrak{P}_k(w)}=\langle\mathbf{h}_w,\mathbf{f}\rangle.
\]
For a fixed $\mu$, we can choose $w$ such that $\langle\mathbf{h}_w,\mathbf{f}\rangle\neq 0$, since $L(\mu,\mathbf{f},\chi) \neq 0$, thanks to the Euler product expansion and absolute convergence for such an $\mu> 2n-1$. Such $\mathbf{h}_{w}$ span $\mathcal{V}(\overline{\Q})$ so for any $\mathbf{h,h'}\in\mathcal{V}(\overline{\Q})$ we have $\langle\mathbf{h,h'}\rangle\in\pi^{\beta}c_k(\mu)^{-1}L(\mu,\mathbf{f},\chi)\overline{\Q}$ and thus $\langle\mathbf{h,h'}\rangle/\langle\mathbf{f,f}\rangle\in\overline{\Q}$. Choose $g_{\infty}$ such that $\mathbf{f}(g_{\mathbf{h}}g_{\infty})\neq 0$. Then
by algebraicity of $\mathbf{f}$, we have
\[
\frac{c_k(\mu)L(\mu,\mathbf{f},\chi)}{\pi^{\beta}\langle\mathbf{f,f}\rangle}=\frac{\langle\mathbf{h}_w,\mathbf{f}\rangle}{\langle\mathbf{f,f}\rangle}\left(\frac{\mathbf{f}(g_{\mathbf{h}}\cdot g_{\infty})}{\mathfrak{P}_k(w)}\right)^{-1}\in\overline{\Q},
\]
and the result follows from the value of $c_k(\mu)$ in Lemma \ref{Reproducing Kernel}.
\end{proof}

We note here that in a forthcoming work of the second named author \cite{YJ1} a construction of a $p$-adic $L$ function attached to Hecke eigenform in the tube case (i.e. $n$ even) is provided. Moreover, in that work, and in the tube case, an improved bound on the weight of the modular form compared to the Theorem \ref{Main Theorem} above is obtained. We also remark that there are other interesting aspects which we have not discussed in this paper as for example the algebraicity of Klingen-type Eisenstein series. This would have made the length of this paper considerably longer. However such aspects are investigated in detail in the Ph.D thesis under progress of the second named author \cite{YJ2}.


\begin{thebibliography}{15}
%

\bibitem{Borel}
A. Borel, Automorphic forms on reductive groups. Automorphic forms and applications, 7–39, IAS/Park City Math. Ser., 12, Amer. Math. Soc., Providence, RI, 2007.

\bibitem{B}
Th. Bouganis, On the standard L-function attached to quaternionic modular forms. J. Number Theory 222 (2021), 293–345.

\bibitem{Deligne} P. Deligne, Travaux de Shimura, Seminaire Bourbaki, 1970/1971, no 389.

\bibitem{G84}
P. Garrett, Imbedded modular curves and arithmetic of automorphic forms on bounded symmetric domains.
Duke Math. J. 51 (1984), no. 2, 431–458.

\bibitem{Harris81}
M.Harris, Maass Operators and Eisenstein Series, Math. Ann. 258, (1981), 135-144.

\bibitem{Harris}
M. Harris, Eisenstein series on Shimura varieties. Ann. of Math. (2) 119 (1984), no. 1, 59–94.

\bibitem{H1}
M. Harris, Arithmetic vector bundles and automorphic forms on Shimura varieties. I. Invent. Math. 82 (1985), no. 1, 151–189.

\bibitem{H2}
M. Harris, Arithmetic vector bundles and automorphic forms on Shimura varieties. II. Compositio Math. 60 (1986), no. 3, 323–378.

\bibitem{Harris94}
M. Harris, $L$-functions and periods of polarized regular motives, J. reine angew. Math 483 (1997), 75-161.

\bibitem{Helgason}
S. Helgason, Differential geometry, Lie groups, and symmetric spaces. Pure and Applied Mathematics, 80. Academic Press, Inc. [Harcourt Brace Jovanovich, Publishers], New York-London, 1978. xv+628 pp. ISBN: 0-12-338460-5

\bibitem{Hua}
L-K. Hua, Harmonic analysis of functions of several complex variables in the classical domains. Translated from the Russian by Leo Ebner and Adam Korányi American Mathematical Society, Providence, R.I. 1963 iv+164 pp.

\bibitem{YJ1} Y. Jin, $p$-adic $L$-functions for quaternionic modular forms, preprint

\bibitem{YJ2} Y. Jin, PhD Thesis, University of Durham, in progress

\bibitem{K}
A. Krieg, Modular forms on half-spaces of quaternions.
Lecture Notes in Mathematics, 1143. Springer-Verlag, Berlin, 1985. xiii+203 pp. ISBN: 3-540-15679-8.

\bibitem{LKW}
K-W. Lan, An example-based introduction to Shimura varieties, to appear in the proceedings of the ETHZ Summer School on Motives and Complex Multiplication.

\bibitem{M}
J. S. Milne, Canonical models of (mixed) Shimura varieties and automorphic vector bundles. Automorphic forms, Shimura varieties, and L-functions, Vol. I (Ann Arbor, MI, 1988), 283–414, Perspect. Math., 10, Academic Press, Boston, MA, 1990.

\bibitem{P}
A. Pitale, A. Saha, R. Schmidt, On the standard L-function for $\mathrm{GSp_{2n}}\times\mathrm{GL}_1$ and algebraicity of symmetric fourth L-values for $\mathrm{GL}_2$. Ann. Math. Quebec 45:113-159 (2021).

\bibitem{PR}
V. Platonov, A. Rapinchuk, Algebraic Groups and Number Theory; Pure and Applied Mathematics Series, vol 139, Academic Press, 1994

\bibitem{Satake}
I. Satake, Algebraic structures of symmetric domains.
Kanô Memorial Lectures, 4. Iwanami Shoten, Tokyo; Princeton University Press, Princeton, N.J., 1980. xvi+321 pp.

\bibitem{PS}
I. I. Pyateskii-Shapiro, Automorphic functions and the geometry of classical domains. Vol. 8 Gordon and Breach Science Publishers, New York-London-Paris 1969 viii+264 pp.

\bibitem{Sh63}
G. Shimura, On analytic families of polarized abelian varieties and automorphic functions. Ann. of Math. (2) 78 (1963), 149–192.

\bibitem{Sh67}
G. Shimura, Algebraic number fields and symplectic discontinuous groups. Ann. of Math. (2) 86 (1967), 503–592.

\bibitem{Sh78}
G. Shimura, Automorphic forms and the periods of abelian varieties, J. Math. Soc. Japan, Vol. 31, No 3, (1979). 

\bibitem{Sh82}
G. Shimura, Confluent hypergeometric functions on tube domains. Math. Ann. 260 (1982), no. 3, 269–302.

\bibitem{Sh84}
G. Shimura, On differential operators attached to certain representations of classical groups. Invent. Math. 77 (1984), no. 3, 463–488.

\bibitem{Sh87}
G. Shimura, Nearly holomorphic functions on Hermitian symmetric spaces. Math. Ann. 278 (1987), no. 1-4, 1–28.

\bibitem{Sh94}
G. Shimura, Differential operators, holomorphic projection, and singular forms. Duke Math. J. 76 (1994), no. 1, 141–173.

\bibitem{Sh95}
G. Shimura, Eisenstein series and zeta functions on symplectic groups. Invent. Math. 119 (1995), no. 3, 539–584.

\bibitem{Sh97}
G. Shimura, Euler products and Eisenstein series. CBMS Regional Conference Series in Mathematics, 93. Published for the Conference Board of the Mathematical Sciences, Washington, DC; by the American Mathematical Society, Providence, RI, 1997. xx+259 pp. ISBN: 0-8218-0574-6.

\bibitem{Sh98}
G. Shimura, Abelian varieties with complex multiplication and modular functions. Princeton Mathematical Series, 46. Princeton University Press, Princeton, NJ, 1998. xvi+218 pp. ISBN: 0-691-01656-9.

\bibitem{Sh00}
G. Shimura, Arithmeticity in the theory of automorphic forms. Mathematical Surveys and Monographs, 82. American Mathematical Society, Providence, RI, 2000. x+302 pp. ISBN: 0-8218-2671-9.

\bibitem{Sh04}
G. Shimura, Arithmetic and analytic theories of quadratic forms and Clifford groups. Mathematical Surveys and Monographs, 109. American Mathematical Society, Providence, RI, 2004. x+275 pp. ISBN: 0-8218-3573-4.

\bibitem{U}
\c{C}.\"{U}rti\c{s}, Special values of $L$-functions by a Siegel-Weil-Kudla-Rallis formula, Journal of Number Theory 125, (2007), 149-181.

\bibitem{JV}
J. Voight, Quaternion algebras. Graduate Texts in Mathematics, 288. Springer, Cham, [2021], ©2021. xxiii+885 pp. ISBN: 978-3-030-56692-0; 978-3-030-56694-4.

\bibitem{Y} S. Yamana, On the lifting of elliptic cusp forms to cusp forms of quaternionic unitary groups, Journal of Number Theory, 130 (2010), 2480-2527

\end{thebibliography}
\end{document}